\documentclass[a4paper,11pt,DIV=11,abstract=on]{scrartcl}
\usepackage[utf8]{inputenc}
\usepackage{fontenc}
\usepackage[english]{babel}
\usepackage{csquotes}
\usepackage{amsmath, amsthm, amssymb,mathtools}
\usepackage{graphicx}
\usepackage{mathrsfs}
\usepackage{subcaption}
\usepackage{xargs}
\usepackage{algorithm}
\usepackage[noend]{algpseudocode}
\usepackage{booktabs}
\usepackage{enumitem}
\usepackage{listings}
\usepackage{environ,ifthen,xcolor}
\usepackage{hyperref}
\usepackage{animate}
\usepackage{yhmath}
\usepackage{bbm}
\mathtoolsset{showonlyrefs}
\definecolor{tuklblue}{RGB}{0,95,140}

\usepackage[bordercolor=tuklblue,linecolor=tuklblue,backgroundcolor=tuklblue!20]{todonotes}
\setlist[enumerate,1]{label={\upshape{\roman*)}},itemsep=.25\baselineskip,topsep=.25\baselineskip}
\setlist[enumerate,2]{label={\upshape{\alph*)}},itemsep=.25\baselineskip,topsep=.25\baselineskip}

\DeclareCaptionLabelSeparator{periodspace}{.\ }
\setkomafont{sectioning}{\rmfamily\bfseries}
\setkomafont{title}{\rmfamily}
\newcommand{\dx}{\mathrm{d}}
\newcommand{\tT}{\mathrm{T}}

\newcommandx{\abs}[2][1=\@empty]{#1\lvert #2 #1\rvert}
\newcommandx{\norm}[3][1=\@empty,3=\@empty]{#1\lVert #2 #1\rVert_{#3}}
\renewcommand{\vec}[1]{\mathbf{#1}}

\newcommand{\grid}{\mathcal{G}}
\newcommand{\HH}{\mathcal{H}}
\newcommand{\NN}{\mathbb{N}}

\DeclareMathOperator*{\argmin}{arg\,min} 
\DeclareMathOperator*{\esssup}{ess\,sup} 

\DeclareMathOperator{\diag}{diag}

\DeclareMathOperator{\trace}{trace}

\DeclareMathOperator{\diffeo}{\mathscr{A}}

\DeclareMathOperator{\SPD}{\mathcal{P}}
\newtheorem{theorem}{Theorem}[section]

\newtheorem{lemma}[theorem]{Lemma}
\newtheorem{remark}[theorem]{Remark}

\newtheorem{corollary}[theorem]{Corollary}

\newboolean{useExternalization}
\setboolean{useExternalization}{false}
\newcommand{\externalFolder}{imgpdf/}
\newcommand{\externalOnly}[1]{
	\ifthenelse{\boolean{useExternalization}}{#1}{}%
}
\NewEnviron{TikZOrPDF}[2][]{
	\ifthenelse{\boolean{useExternalization}}{%
		\tikzsetnextfilename{#2}%
		\BODY
	}{%
	\ifthenelse{\equal{#1}{}}{
		\includegraphics{\externalFolder/#2.pdf}%
	}{
	\includegraphics[width=#1]{\externalFolder#2.pdf}%
}%
}%
}{}%

\externalOnly{
	\usepackage{tikz,pgfkeys,pgfplots,tikz-3dplot}
	\usetikzlibrary{external}
	\tikzexternalize[prefix=\externalFolder]
	\usetikzlibrary{spy,calc,external,arrows}
	\usetikzlibrary{decorations.markings,positioning}
	\usepackage{letltxmacro}
	\tikzstyle{point}=[inner sep=3ptpt, outer sep=0pt,fill=black]%
	
	\newenvironment{customlegend}[1][]{%
		\begingroup
		\csname pgfplots@init@cleared@structures\endcsname
		\pgfplotsset{#1}%
	}{%
	\csname pgfplots@createlegend\endcsname
	\endgroup
}%

\def\addlegendimage{\csname pgfplots@addlegendimage\endcsname}
\pgfplotsset{
	/pgfplots/uxbox/.style ={
		legend image code/.code={
			\draw[thin] (-0.08cm,-0.08cm) rectangle ++(0.16cm,0.16cm); 	
			\draw[fill] (0.4,-0.08cm) rectangle ++(0.16cm,0.16cm); 		
			\draw[->] (0,0)-- ++(0.3cm,0);
			\draw[->] (0.48,0)-- ++(0.3cm,0);
		}
	}
}
\pgfplotsset{
	/pgfplots/uybox/.style ={
		legend image code/.code={
			\draw[thin] (-0.08cm,-0.08cm) rectangle ++(0.16cm,0.16cm);
			\draw[->] (0,0)-- ++(0,0.3cm);			
			\draw[fill] (0.4,-0.08cm) rectangle ++(0.16cm,0.16cm); 			
			\draw[->] (0.48,0)-- ++(0,0.3cm);
		}
	}
}

\pgfkeys{/pgfplots/number in legend/.style={%
		/pgfplots/legend image code/.code={%
			\node at (0.295,-0.0225){#1};
		},%
	},
}
}

\DeclareMathOperator{\RR}{\mathbb{R}}

\hypersetup{pdfauthor={Sebastian Neumayer, Johannes Persch, Gabriele Steidl},
	pdftitle={Morphing of Manifold-Valued Images inspired by Discrete Geodesics in Image Spaces },
	pdfsubject={Draft},%
	pdfcreator = {pdflatex and TexStudio},%
	pdfkeywords={},
	plainpages=false, pdfstartview=FitH, pdfview=FitH, pdfpagemode=UseOutlines,%
	bookmarksnumbered=true,bookmarksopen=false,bookmarksopenlevel=0,%
	colorlinks=true,linkcolor=black,citecolor=black,urlcolor=black%
}
\begin{document}

\title{Morphing of Manifold-Valued Images\\
	inspired by Discrete Geodesics in Image Spaces 
}

\author{
	Sebastian Neumayer\footnotemark[1] \and Johannes Persch\footnotemark[1] \and Gabriele Steidl\footnotemark[1] \footnotemark[2]}

\maketitle
\footnotetext[1]{Department of Mathematics,
	Technische Universität Kaiserslautern,
	Paul-Ehrlich-Str.~31, D-67663 Kaiserslautern, Germany,
	\{sneumaye,persch,steidl\}@mathematik.uni-kl.de.} 
\footnotetext[2]{Fraunhofer ITWM, Fraunhofer-Platz 1,
	D-67663 Kaiserslautern, Germany}
\begin{abstract}
	This paper  addresses the morphing of manifold-valued images 
	based on the  time discrete geodesic paths model
	of Berkels, Effland and Rumpf~\cite{BER15}.
	Although for our manifold-valued setting such an interpretation of the 
	energy functional is not available so far,
	the model is interesting on its own.
	We prove the existence of  a minimizing sequence within the set of $L^2(\Omega,\HH)$
	images having values in a finite dimensional Hadamard manifold $\HH$ 
	together with a minimizing sequence of admissible diffeomorphisms.
	To this end, we show that the continuous manifold-valued functions are dense in $L^2(\Omega,\HH)$.
	We propose a space discrete model based on a finite difference approach on staggered grids,
	where we focus on the linearized elastic potential in the regularizing term.
	The numerical minimization alternates between 
	i) the computation of a deformation sequence between given images
	via the parallel solution of certain registration problems
	for manifold-valued images, and ii) the computation of an image sequence 
	with fixed first (template) and last (reference) frame
	based on a given sequence of deformations via the solution 
	of a system of equations arising from the corresponding Euler-Lagrange equation.
	Numerical examples give a proof of the concept of our ideas.
\end{abstract}

\section{Introduction}

Smooth image transition, also known as image morphing, is a frequently addressed
task in image processing and computer vision, and there are various approaches to tackle the problem. 
For example, in feature based morphing only specific features are mapped to each other 
and the whole deformation is then calculated by interpolation.
This was successfully applied, e.g., in the production of the movie \emph{Willow}~\cite{Smythe1990}. 
We refer to  \cite{wolberg1990,Wolberg1998} for an overview of similar techniques.
This paper is related to
a special kind of image morphing, the so-called metamorphosis introduced by 
Miller, Trouv\'e and Younes \cite{MY2001,TY2005a,TY2005b}.
The metamorphosis model can be considered 
as an extension of the flow of diffeomorphism model and its large deformation diffeomorphic metric mapping framework \cite{BMTY2005,CRM96, DGM98,Tro95,Tro98}
in which \emph{each} image pixel is transported along a trajectory determined by a diffeomorphism path. 
As an extension the metamorphosis model allows the variation of image intensities along trajectories of the pixels.
Solutions via shooting methods were developed e.g. in \cite{HJS12,RY16}, where the first reference considers a  metamorphosis regression model.
A comprehensive overview over the topic is given in the book \cite{Younes2010} as well as in the review article \cite{MTY15}.
For a historic account see also \cite{MTY02}.

This paper builds up on a time discrete geodesic paths model
by Berkels, Effland and Rumpf \cite{BER15}.
We mention that such a variational time discretization in shape spaces was already used in \cite{RW13,RW15}, see also \cite{FJSY09}.
Let $\Omega \subset \mathbb R^n$, $n \ge 2,$ be an open, bounded, connected domain with Lipschitz boundary.
The authors of \cite{BER15} define a (time) \emph{discrete geodesic} 
connecting a template image $I_0 \coloneqq T \in L^2(\Omega,\mathbb R)$ and 
a reference image $I_K\coloneqq R\in L^2(\Omega,\mathbb R)$, $K \ge 2,$
as minimizing sequence $\vec I = (I_1,\ldots,I_{K-1}) \in \left( L^2(\Omega,\mathbb R) \right)^{K-1}$ 
of the \emph{discrete path energy}
\begin{align} \label{mdr}
{\boldsymbol{\mathcal J}}_{\mathrm{\small{BER}}}(\vec I)
\coloneqq  
&\sum_{k=1}^K  \inf_{\varphi_k\in \diffeo} 
\int_\Omega W(D\varphi_k(x))+\gamma\lvert D^m\varphi_k(x)\rvert^2\dx x
+ \frac{1}{\delta} \int_{\Omega} (I_{k} \circ \varphi_{k} - I_{k-1})^2 \dx x, \\
&\mbox{subject to} \quad I_0 = T, \; I_K = R,	\nonumber
\end{align}
where $\delta, \gamma > 0$,  the function $W$ has to satisfy certain properties,
$\diffeo$ is an admissible set of deformations, and
the higher order derivatives $D^m \varphi_k$, $m> 1 + \frac{n}{2}$, guarantee a certain smoothness of the deformation.
Berkels, Effland, and Rumpf showed that under certain constrains on $W$ 
minimizers of ${\boldsymbol{\mathcal J}}_{\mathrm{\small{BER}}}$ converge for $K \rightarrow \infty$ to a minimizer
of the continuous geodesic path model of Trouv\'e and Younes~\cite{TY2005a,TY2005b}, where the deformation is regularized by the dissipation density of a Newtonian fluid, which is the linearized elastic potential applied to the time derivative of the diffeomorphism path.

In this paper, we generalize the model \eqref{mdr} to manifold-valued images in $L^2(\Omega,\HH)$
and prove that it is well-defined at least for finite dimensional Hadamard manifolds $\HH$.
These are simply connected, complete Riemannian manifolds with non-positive sectional curvature.
Typical examples of such Hadamard manifolds are hyperbolic spaces and 
symmetric positive definite matrices with the affine invariant metric.
As an important fact we will use that the distance in Hadamard spaces is jointly convex
which implies weak lower semicontinuity of certain functionals involving the distance function.
Since we use another admissible set than in \cite{BER15} all proofs are new also for real-valued images.
So far we have not established a relation of our model to some kind of time continuous path energy model
in the image space $L^2(\Omega,\HH)$.

Dealing with digital images we have to introduce a space discrete model. 
In contrast to the finite element approach in \cite{BER15},
we prefer a finite difference model on a staggered grid.
We have used this discretization for gray-value images in \cite{PPS17}.
For finding a minimizer, we propose an alternating algorithm
fixing either the deformation or the image sequence:
\begin{itemize}
 \item[i)]
For a fixed image sequence, we have to solve certain registration problems for \emph{manifold-valued images} 
in parallel to get a sequence $(\varphi_1,\ldots,\varphi_K)$ of deformations. 
Necessary interpolations were performed via Karcher mean computations.
There exists a rich literature on registration problems for images with values in the Euclidean space, 
see e.g.~\cite{CJ2001,FM2003,HM06,HBDHRSSU2007} 
and, for an overview,  the books of Modersitzki \cite{Mod2004,Mod2009}.
 \item[ii)] For a fixed deformation sequence,
we need to find a minimizing image sequence $(I_1,\ldots,I_{K-1})$ of
\[
\sum_{k = 1}^{K}  \mathrm{d}_2^2(I_{k}\circ \varphi_{k},I_{k-1}) \quad \mbox{subject to} \quad I_0 = T, I_K = R
\]
where $\mathrm{d}_2$ denotes the distance in $L^2(\Omega,\HH)$.
In our manifold-valued setting, this requires to evaluate geodesics between manifold-valued image pixels 
at several well-defined time steps.
\end{itemize}

\paragraph{Outline of the Paper:}
We start with preliminaries on Hadamard spaces in Subsection \ref{sec:non_leb}, where
the focus is on the proof that the uniformly continuous functions 
mapping into locally compact Hadamard spaces $\HH$ are dense in $L^p(\Omega,\HH)$, $p \in [1,\infty)$. 
We have not found this result in the literature.
We introduce the necessary notation in Sobolev and Hölder spaces in Subsection \ref{sec:admiss}.
The important definition is those of an admissible set of deformations which differs from that in
\cite{BER15}. In particular, our definition guarantees
that the  concatenation of an image $I \in L^2(\Omega,\HH)$ with a deformation from our admissible set $I \circ \varphi$ 
is again an image in
$L^2(\Omega,\HH)$.
In Section \ref{sec:dm}, we introduce our space continuous model for manifold-valued images
and prove the existence of minimizers.
Although we could roughly follow the lines in \cite{BER15}, all proofs are new 
also for the Euclidean setting $\HH \coloneqq \mathbb R^d$ due to the different admissible set.
Moreover, the manifold-valued image setting requires 
a nontrivial update of the nice regridding idea from \cite{BER15,RW13,RW15} in Theorem \ref{lem:uni:seq}.
This influences also the proof of the main result in Theorem \ref{main}.
In Section \ref{sec:d_model}, we detail the computation issues for the space discrete model,
where we propose a finite difference scheme on a staggered grid together with a multiscale technique.
Numerical examples are given in Section \ref{numerics}. 
Note that the numerical algorithms are not restricted to Hadamard manifolds.
We finish with conclusions in Section \ref{sec:conclusions}.

\section{Preliminaries}\label{sec:prelim}
	
Throughout this paper, let $\Omega \subset\RR^n$, $n \ge 2,$ be an open, bounded, connected domain with Lipschitz boundary. 
	
\subsection{Hadamard Spaces}\label{sec:non_leb}

For an overview on Hadamard spaces we refer to the books \cite{Bac14,BH1999,Jost97}.
First, recall that a metric space $(X,d)$ is \emph{geodesic} if every two points $x,y \in X$  
are connected by a curve $\gamma_{\overset{\frown}{x,y}}\colon [0,1] \to X$, called geodesic, such that
 \begin{equation}\label{eq:geo}
  d \bigl(\gamma_{\overset{\frown}{x,y}}(s),\gamma_{\overset{\frown}{x,y}}(t) \bigr) 
  = \lvert s - t\rvert d \bigl(\gamma_{\overset{\frown}{x,y}}(0),\gamma_{\overset{\frown}{x,y}}(1) \bigr), 
	\qquad \text{for every } s,t \in [0,1],
 \end{equation}
 and $\gamma_{\overset{\frown}{x,y}}(0)=x$ and $\gamma_{\overset{\frown}{x,y}}(1)=y$.
A complete metric space $(\HH,d)$ is called a \emph{Hadamard space} if it is geodesic 
and if for every geodesic triangle $\triangle p,q,r \in  \HH$ 
and 
$x \in \gamma_{\overset{\frown}{p,r}}$,
$y \in \gamma_{\overset{\frown}{q,r}}$
we have $d(x,y) \le |\bar x - \bar y|$, where $\bar x, \bar y$ are corresponding points
in the comparison triangle $\triangle \bar p, \bar q,\bar r \in \mathbb R^2$ 
having the same side lengths as the geodesic one.
By~\cite[Theorem 1.1.3]{Bac14} this is equivalent to $(\HH,d)$ being 
a complete metric geodesic space with
 \begin{equation}\label{eq:reshet0}
 d^2(x,v) + d^2(y,w) \le d^2(x,w) + d^2(y,v) + 2d(x,y)d(v,w),
 \end{equation}
 for every $x,y,v,w \in \HH$. 
 Inequality~\eqref{eq:reshet0} implies that geodesics are uniquely determined by their endpoints.
 Later we will restrict our attention to finite dimensional Hadamard manifolds which are Hadamard spaces
 having additionally a Riemannian manifold structure. 
  
A function $f\colon \HH \rightarrow \mathbb R$ is called \emph{convex} if for every $x,y \in \HH$ the function 
$f \circ \gamma_{\overset{\frown}{x,y}}$ is
convex, i.e., if
\begin{equation*}
	f\bigl( \gamma_{\overset{\frown}{x,y}}(t) \bigr)
	\le (1-t) f\bigl( \gamma_{\overset{\frown}{x,y}} (0) \bigr)
	+ t f \bigl(\gamma_{\overset{\frown}{x,y}}(1)\bigr),
\end{equation*}
for each $t \in [0,1]$. 
An important property of Hadamard spaces, which is also fulfilled in more general Busemann spaces, 
is that the distance is jointly convex, see~\cite[Proposition 1.1.5]{Bac14}, i.e., 
for two geodesics $\gamma_{\overset{\frown}{x_1,x_2}},\gamma_{\overset{\frown}{y_1,y_2}}$ and $t\in[0,1]$ it holds
\begin{equation}
\label{eq:jointconvexity}
d\bigl(\gamma_{\overset{\frown}{x_1,x_2}}(t),\gamma_{\overset{\frown}{y_1,y_2}}(t)\bigr) \le (1-t)d(x_1,y_1)+td(x_2,y_2).
\end{equation}
For a bounded sequence $\{x_n\}_{n\in\NN}$ of points $x_n \in \HH$, the function 
$w \colon \HH \to [0, +\infty)$ defined by
\begin{equation*}
  w(x;\, \{x_n\}_{n\in\NN}) \coloneqq \limsup_{n\to\infty} d^2(x, x_n)
\end{equation*}
has a unique minimizer, which is called the \emph{asymptotic center}
of $\{x_n\}_{n\in\NN}$, see \cite[p.~58]{Bac14}.
A sequence $\{x_n\}_{n\in\NN}$ \emph{weakly converges} to a point $x \in \HH$ if it is bounded 
and $x$ is the asymptotic center of each subsequence of $\{x_n\}_{n\in\NN}$, see  \cite[p.~103]{Bac14}. 
The definition of proper and (weakly) lower semicontinuous (lsc) functions carries over from the Hilbert space setting.
 	
	On $\HH$ we consider the Borel $\sigma$-algebra $\mathcal{B}$.
	A function $f\colon\Omega\to\HH$ is (Lebesgue) measurable
	if $\{\omega \in \Omega: f(\omega) \in B \}$ is a (Lebesgue) measurable set for all $B \in \mathcal{B}$.
	In the following, we only consider the Lebesgue measure $\mu$ on $\mathbb R^n$.
	A  measurable map $f\colon\Omega\to\HH$ 
	belongs to $\mathcal{L}^p(\Omega,\HH)$, $p \in [1,\infty]$, if 
	\begin{equation*}
		\mathrm{d}_p(f(\omega),a) < \infty,
	\end{equation*}
	for any constant mapping $\omega \mapsto a$ to a fixed $a\in \HH$, where
	$\mathrm{d}_p$ is defined for two measurable maps $f$ and $g$ by
	\begin{equation*}
	\mathrm{d}_p(f(\omega),g(\omega)) \coloneqq 
	\begin{cases}
		\Bigr(\int_{\Omega}d^p(f(\omega),g(\omega)) \, \dx \omega \Bigr)^{\frac{1}{p}} \quad & p\in[1,\infty),\\
		\esssup_{\omega\in\Omega}d(f(\omega),g(\omega)) & p  =\infty.
	\end{cases}
	\end{equation*}
	 With the equivalence relation
	$	f\sim g$ if $\mathrm{d}_p(f,g) = 0$,
	the quotient space $L^p(\Omega,\HH)\coloneqq \mathcal{L}^p(\Omega,\HH)/ \sim$ equipped with $\mathrm{d}_p$ becomes a complete metric space.
	For $p=2$ it is again a Hadamard space, see \cite[Proposition 1.2.18]{Bac14}.

	By $C(\Omega,\HH)$ we denote the space of continuous maps from $\Omega$ to $\HH$.
	Next we show that $C(\Omega,\HH)$ is dense in $L^p(\Omega,\HH)$, 
	more precisely, also the uniformly continuous functions are dense.
	We start by defining simple functions (step functions). A function $g \in L^p(\Omega,\HH)$
	is called a \emph{simple function} if there exists a finite partition of 
	$\Omega =  \dot\bigcup_{i \in \mathcal{I}} A_i$ into disjoint measurable sets $A_i$ such that
	$g|_{A_i} = a_i$ for all $i \in \mathcal{I}$.
	There exists a Hopf-Rinow-like theorem for Hadamard spaces which says that in locally compact Hadamard spaces closed and bounded sets are compact, 
	see 	\cite[p. 35]{BH1999}.

	\begin{lemma}\label{lem:dense_help_1}
	Let $(\HH,d)$ be a locally compact Hadamard space.
	Then the simple functions are dense in $L^p(\Omega,\HH)$, $p \in [1,\infty)$.
	\end{lemma}
	
	\begin{proof}
	Let $f \in L^p(\Omega,\HH)$. 
	Then we have for a fixed reference point $a \in \HH$ that
	$$
	I_N \coloneqq \int\limits_{\{\omega \in \Omega: d( f(\omega),a) > N \} } d^p( f(\omega),a) \, \dx \omega \rightarrow 0
	$$
	as $N \rightarrow \infty$.
	For an arbitrary $\varepsilon > 0$, we choose $N = N(\varepsilon)$ such that $I_N < \frac{\varepsilon}{2}$ and set
	$$
	\mathcal{A}_0  \coloneqq  \{x \in \HH : d(x,a) > N\} 
	\quad {\textrm and} \quad
	A_0 \coloneqq \{\omega \in \Omega: f(\omega) \in  \mathcal{A}_0\}.
	$$
	Next we cover the compact set
	$$
	\mathcal{A} \coloneqq \{x \in \HH : d(x,a) \le N\}.
	$$
	with open balls of radius $r_\varepsilon^p\coloneqq\frac{1}{2^p}\frac{\varepsilon}{2 \mu(\Omega)}$.
	Since $\mathcal{A}$ is compact, this covering contains a finite subcovering
	which can be restricted to $\mathcal{A}$ and made disjoint so that
	$\mathcal{A} = \dot \cup_{i =1}^M \mathcal{A}_i$ for some $M \in \mathbb N$.
	Fixing any $a_i \in \mathcal{A}_i$, 
	we have 
	$d(x, a_i) \le 2r_\varepsilon$ 
	for all $x \in \mathcal{A}_i$.
	Since the $\mathcal{A}_i \in \mathcal{B}$ and $f$ is measurable, 
	the sets 
	$A_i \coloneqq \{\omega \in \Omega: f(\omega) \in  \mathcal{A}_i\}$
	are measurable.
	Thus, $\Omega = \dot \cup_{i =0}^M A_i$ is a finite disjoint partition of $\Omega$ into measurable sets.
	Defining the simple function $g\colon \Omega \rightarrow \HH$ by
	$g|_{A_i} \coloneqq a_i$, $i=0,\ldots,M$, where $a_0 \coloneqq a$,
	and conclude
	\begin{align*}
	\int_{\Omega} d^p \left(f(\omega),g(\omega) \right) \, \dx \omega 
	&= 
	\int_{A_0} d^p \left(f(\omega),a_0 \right) \, \dx \omega + \sum_{i=1}^M \int_{A_i} d^p \left(f(\omega),a_i \right) \, \dx \omega \\
	&\le
	\frac{\varepsilon}{2} + \sum_{i =1}^M \mu(A_i) \frac{\varepsilon}{2 \mu(\Omega)} \le \varepsilon.
	\end{align*}
	\end{proof}
	\begin{figure}[t]\label{fig:Helper}
		\begin{subfigure}[b]{0.49\textwidth}
			\centering
			\begin{TikZOrPDF}{sets}
			\begin{tikzpicture}[scale=0.37]
			\draw[draw=none,fill,white!75!black] (-10,4) -- (0,4) -- (0,-4) -- (-10,-4) -- (-10,-3) -- (-1,-3) -- (-1,3) -- (-10,3);	
			\draw[draw=none,fill,white!75!black] (7,4) -- (0,4) -- (0,-4) -- (7,-4) -- (7,-3) -- (1,-3) -- (1,3) -- (7,3);
			
			\draw[draw=black, dashed,fill = white!82!black]  (7,-3) -- (1,-3) -- (1,3) -- (7,3);
			\draw[draw=black, dashed,fill = white!82!black] (-10,3) -- (-1,3) -- (-1,-3) -- (-10,-3);
			
			\draw[draw=black,fill = white!90!black]  (7,-2) -- (2,-2) -- (2,2) -- (7,2);
			\draw[draw=black,fill = white!90!black] (-10,2) -- (-2,2) -- (-2,-2) -- (-10,-2);
			\draw[dashed] (-10,4) -- (7,4);
			\draw[dashed] (-10,-4) -- (7,-4);
			\draw[dashed] (0,4) -- (0,-4);

			\draw (-10,2) -- (-2,2) -- (-2,-2) -- (-10,-2);
			\draw (7,-2) -- (2,-2) -- (2,2) -- (7,2);
			
			\draw[fill] (0,0) circle(2pt);\
			\draw[fill] (-5 , 0) circle(2pt);
			
			\draw node[anchor=west] at (-10 , -3.5) {$A_i$};
			\draw node[anchor=west] at (-10 , -2.5) {$B_i$};
			\draw node[anchor=west] at (-5 , 0) {$a_i$};
			\draw node[anchor=east] at (-7 , 0) {$\mathcal{K}_i$};
			
			\draw node[anchor=east] at (7 , -3.5) {$A_j$};
			\draw node[anchor=east] at (7 , -2.5) {$B_j$};
			\draw node[anchor=east] at (5 , 0) {$\mathcal{K}_j$};
			\draw node[anchor=west] at (0 , 0) {$a_0$};
			
			\end{tikzpicture}
			\end{TikZOrPDF}
			\caption{Illustration for the construction of the $B_i$ and $\mathcal K_i$.}
			\label{subfig:Helper1}
		\end{subfigure}
		\begin{subfigure}[b]{0.49\textwidth}
		\centering
		\begin{TikZOrPDF}{setsplot}
		\begin{tikzpicture}[scale=0.75]
		\begin{axis}
		[
		xtick 				= {-0.5,-0.25,0},
		xticklabels			= {$\frac{1}{2}$,$\frac{1}{4}$,$0$},
		ytick 				= {},
		yticklabels			= {},
		xmax                = 0.75, 
		xmin                = -0.75, 
		ymax                = 1.1, 
		ymin                = 0.0, 
		hide y axis,
		axis x line=middle,
		x label style={at={(ticklabel cs:.75)},anchor=near ticklabel,below=-3mm},
		x axis line style={stealth-},
		anchor=origin,
		yscale = 0.55,		
		xlabel = {dist$(w,B_i^c)$}
		]
		\addplot[red,mark=none,smooth,thick,domain = -0.75:0] {0.5};
		\addplot[red,mark=none,smooth,thick,domain = 0:1] {1};
		\addplot[black,mark=none,smooth,dashed,thick,domain = -0.4975:0] {1+x};
		\addplot[black,mark=none,smooth,thick,domain = -0.2525:0] {1+2*x};
		\addplot+[black,mark=none,smooth,thick,domain = -0.75:-0.25] {0.4885};
		\addplot+[black,mark=none,smooth,thick,domain = -0.75:-0.5] {0.5115};
		
		\draw[red] node at (axis cs:-0.005,0.5) {$)$};
		\draw[red] node at (axis cs:0.005,1) {$[$};
		\draw node[anchor=west] at (axis cs:-0.125,0.7) {$h_4$};
		\draw node[anchor=east] at (axis cs:-0.375,0.7) {$h_2$};
		\draw node[anchor=north] at (axis cs:0.5,1) {\color{red}$\tilde{g}$};
		\end{axis}
		\end{tikzpicture}
		\end{TikZOrPDF}
		\caption{Illustration for the construction of the uniformly continuous approximation $h_k$.}
		\label{subfig:Helper2}
		\end{subfigure}		
	\caption{Illustrations to the proof of Theorem~\ref{thm:dense}.}
	\end{figure}
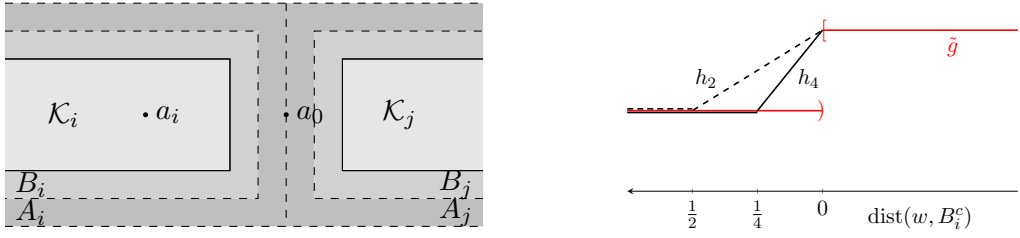
	
	\begin{theorem}\label{thm:dense}
	Let $(\HH,d)$ be a locally compact Hadamard space.
	Then the set of uniformly continuous functions mapping from $\Omega$ to $\HH$ is dense in $L^p(\Omega,\HH)$, $p \in [1,\infty)$.
	\end{theorem}
	
	\begin{proof}
	By Lemma \ref{lem:dense_help_1}, it suffices to show that simple functions can be well approximated
	by uniformly continuous functions.
	
	Let $g\colon \Omega \rightarrow \HH$ be a simple function
	determined by $A_i$ and corresponding $a_i \in \HH$, $i \in \mathcal{I}$ from a finite index set $\mathcal{I} \subset \mathbb N$.
	Set
	$C \coloneqq \max_{i,j \in \mathcal{I}} d^p(a_i,a_j)$.
	Let $\varepsilon > 0$ be arbitrary small.
	Since the Lebesgue measure is regular there exist compact sets ${\mathcal K}_i \subseteq A_i$ such that
	$$\mu(A_i\backslash {\mathcal K}_i) \le \left( \frac{\varepsilon}{2} \right)^p \frac{1}{C \, |\mathcal{I}|}.$$
	Since the ${\mathcal K}_i$ are disjoint and compact, there exists $\delta >0$ such that
	$\mathrm{dist} ({\mathcal K}_i, {\mathcal K}_j) > \delta$ and $\mathrm{dist}(\mathcal{K}_i,\partial \Omega)> \delta$ for all $i,j \in \mathcal{I}$, 
	where $\mathrm{dist}$ is the distance function with respect to the Euclidean
	norm in $\mathbb R^n$.
	Then the open sets 
	$B_i = B_{\frac{\delta}{4}} ({\mathcal K}_i) \coloneqq \{\omega \in  \Omega: \mathrm{dist}(\omega, {\mathcal K}_i) < \frac{\delta}{4} \}$,  
	$i \in \mathcal{I}$, are disjoint.
	Let $B_0 \coloneqq \Omega \backslash \cup_{i \in \mathcal{I}} B_i$ and $a_0 \coloneqq a_1$.
	We define a simple function $\tilde g\colon\Omega \rightarrow \HH$ by
	$\tilde g\vert_{B_i} \coloneqq a_i$, $i \in \{0\} \cup \mathcal{I}$. For an illustration see Fig.~\ref{subfig:Helper1}.
	It is related to the original step function $g$ by
	\begin{equation} \label{ab1}
	\mathrm{d}^p_p(g,\tilde g) = \int_{\Omega \backslash \bigcup_{i \in \mathcal{I}} {\mathcal K}_i} d^p \left(g(\omega) ,\tilde g(\omega) \right) \, \dx \omega
	\le C \, \sum_{i \in \mathcal{I}} \mu(A_i \backslash {\mathcal K}_i)  \le \left( \frac{\varepsilon}{2} \right)^p.
	\end{equation}
	Next we approximate $\tilde g$ by a sequence of uniformly continuous functions.
	Let $\gamma_{\overset{\frown}{a,b}}$ denote the unique geodesic joining $a,b \in \HH$, where
	$\gamma_{\overset{\frown}{a,b}}(0) = a$ and $\gamma_{\overset{\frown}{a,b}}(1) = b$.
	For $k \in \NN$, define $h_k\colon \Omega \rightarrow \HH$ by
	$h_k|_{B_0} \coloneqq a_0$ and for $i \in \mathcal{I}$,
	$$
	h_k(\omega) \coloneqq \gamma_{\overset{\frown}{a_0,a_i}} \left( \min\{1, k \mathrm{dist} (  B_i^c,\omega)\} \right), \quad \omega \in B_i,
	$$
	see Fig.~\ref{subfig:Helper2}.
	Since $\gamma$ is a geodesic, 
	$h_k$ is  by construction continuous on every $B_i$, $i \in \mathcal{I}$. Further, $h_k|_{\partial B_i} = a_0$
	and for any sequence $\{ \omega_j \}_{j\in\NN}$ converging to some $\hat \omega \in \partial B_i$, we have that $h_k(\omega_j) $ converges to 
	$a_0 = h_k(\hat \omega)$. Hence $h_k$ is continuous on $\Omega$ and by construction constant outside 
	of the compact set $\cup_{i\in\mathcal{I}}\overline{B}_i$. This implies, that $h_k$ is even uniformly continuous on $\Omega$.
	Let
	$$
	\max_{i,j \in \mathcal{I}} \sup_{x \in \gamma_{\overset{\frown}{a_0,a_i}} } d^p(x,a_j) \le D.
	$$
	Then we get
	$
	d^p\left(\tilde g(\omega),h_k(\omega) \right) \le D$ for all $k \in \mathbb N,\omega\in\Omega,$ 
	so that $D$ is an integrable bound of $d^p(\tilde g,h_k)$
	and
	$
	\lim_{k \rightarrow \infty} d\left(\tilde g(\omega),h_k(\omega) \right)^p \rightarrow 0 
	$
	pointwise as $k \rightarrow \infty$. By Lebesgue's convergence theorem this implies that
	$$
	\lim_{k\rightarrow \infty} \mathrm{d}_p(\tilde g, h_k) = 0.
	$$
	Finally, choosing $k \in \mathbb N$ such that $\mathrm{d}_p(\tilde g, h_k) < \frac{\varepsilon}{2}$ we obtain with
	\eqref{ab1} that
	$$
	\mathrm{d}_p(g,h_k) \le \mathrm{d}_p(g,\tilde g) + \mathrm{d}_p(\tilde g, h_k) < \varepsilon.
	$$
	This finishes the proof.
	\end{proof}
	
 	In Hadamard spaces we have no zero element so that 
 	a notion of ``compact support'' is not available. 
 	However, the previous proof shows that we can 
 	approximate a function in $L^p(\Omega, \HH)$ arbitrarily well by a function 
 	which is constant outside of a compact set which can be seen as an equivalent. The previous theorem can be applied to prove the next corollary.

 \begin{corollary}\label{stet_norm}
  Let $(\HH,d)$ be a locally compact Hadamard space.\\
 	Let $f \in L^p(\Omega,\HH)$, $p \in [1,\infty)$ and $\{ \varphi^{(j)}\}_{j \in \NN}$, 
 	be a sequence of diffeomorphisms on $\Omega$ such that 
	$\lim_{j \to \infty}\Vert \varphi^{(j)} - \hat \varphi \Vert_{(L^\infty(\Omega)^n} = 0$
	and $|\mathrm{det} (D \varphi^{(j)}) \vert^{-1} \leq C$ for all $j \in \NN$. 
	Then $\limsup_{j \to \infty} \mathrm{d}_p(f\circ \varphi^{(j)}, f \circ \hat \varphi) = 0$.
 \end{corollary}
 
 \begin{proof}
 	For $f\in L^p(\Omega,\HH)$,  we can construct  by Theorem \ref{thm:dense}
 	a  sequence $\{ f_k\}_{k\in\NN}$ of uniformly continuous functions
 	with $\mathrm{d}_{p}(f,f_k)<\frac{1}{k}$. 
	Then we conclude
 	\begin{align*}
 	\mathrm{d}_p(f\circ \varphi^{(j)}, f \circ \hat \varphi)
	&\leq 
	\mathrm d_p(f\circ \varphi^{(j)}, f_k\circ \varphi^{(j)}) +  \mathrm d_p(f_k\circ \varphi^{(j)}, f_k \circ \hat \varphi) 
	+  \mathrm d_p(f_k \circ \hat \varphi, f \circ \hat \varphi)
 	 \\
	&\leq 
	C \mathrm d_p(f, f_k) +  \mathrm d_p(f_k\circ \varphi^{(j)}, f_k \circ \hat \varphi) +  C \mathrm d_p(f_k, f).
 	\end{align*}
 	By construction the first and the last term can be made arbitrary small as $k \rightarrow \infty$.
 	Now let $k \in \NN$ be fixed. Since $\varphi^{(j)}$ converges uniformly to $\hat \varphi$, we can use the uniform continuity of $f_k$ 
 	in order to conclude that $f_k\circ \varphi^{(j)}$ converges uniformly to $f_k \circ \hat \varphi$. 
 	Now boundedness of $\Omega$ implies that the middle term converges to zero as $j \to\infty$.
 \end{proof} 
 
        In the rest of this paper, we restrict our attention to finite dimensional Hadamard manifolds $(\HH,d)$,
	i.e., $\HH$ has an additional Riemannian manifold structure. 
	Clearly finite dimensional Hadamard manifolds are locally compact.
	By $T_x \HH$ we denote the tangential space of $\HH$ 	at $x \in \HH$.
	Then the geodesics $\gamma_{x,v}$  are determined by their starting point $x \in \HH$ and their tangential $v \in T_x \HH$ 
	at this point.
	We will need the exponential map 
	$\exp_x\colon T_x \HH \rightarrow \HH$ 
	defined by $\exp_x v \coloneqq \gamma_{x,v} (1)$, and the inverse of the exponential map
	$\log_x \coloneqq \exp_x^{-1}\colon \HH \rightarrow T_x \HH$.

\subsection{Sobolev Spaces and Admissible Mappings}\label{sec:admiss}

	Let $C^{k,\alpha}(\overline \Omega)$, $k \in \mathbb N_0$,  denote the H\"older space of functions $f \in C^k(\overline \Omega)$
	for which 
	$$
	\|f\|_{C^{k,\alpha}(\overline \Omega)} 
	\coloneqq \sum_{|\beta| \le k} \|D^\beta f\|_{C(\overline \Omega)} 
	+  \sum_{|\beta| = k}   \sup_{\substack{x,y \in \Omega \\ x \not = y}} \frac{|D^\beta f(x) - D^\beta f(y)|}{|x-y|^\alpha}
	$$
	is finite. With this norm $C^{k,\alpha}(\overline \Omega)$ is a Banach space.
	
	By $W^{m,p}(\Omega)$, $m \in \mathbb N$, $1 \le p < \infty$, we denote the Sobolev space of functions having weak derivatives up to order $m$ 
	in $L^p(\Omega)$ with norm
	\begin{equation*}
	\lVert f\rVert_{W^{m,p}(\Omega)}^p \coloneqq \int_{\Omega}\sum_{\lvert\alpha\rvert \le m} \lvert D^{\alpha} f(x)\rvert^p \dx x.
	\end{equation*}
	We apply the usual abbreviation
	$|D^m f|^p \coloneqq \sum_{\lvert\alpha\rvert = m} \lvert D^{\alpha} f\rvert^p$.
	For $F = (f_\nu)_{\nu=1}^n$, we set 
	$
	|D^m F|^p = \sum_{\nu =1}^n \lvert D^m f_\nu \rvert^p
	$.
	In particular, we are interested in $W^{m,2}(\Omega)$ with
	$m> 1 + \frac{n}{2}$.
	In this case,  $W^{m,2}(\Omega)$ is compactly embedded in 
	$C^{1,\alpha}(\overline \Omega)$ for all $\alpha \in (0,m-1-\frac{n}{2})$ \cite[p.~350, Th.~8.13]{alt2002}
	and consequently $W^{m,2}(\Omega)  \hookrightarrow W^{1,p}(\Omega)$ for all $p \ge 1$.

	For $m> 1 + \frac{n}{2}$, we consider the set
	\begin{equation*}
	\diffeo \coloneqq \{\varphi\in \left(  W^{m,2}(\Omega) \right)^n: \det(D\varphi) > 0 
	\text{ a.e.~in } \Omega, \; \varphi(x) = x  \; \mathrm{for} \; x \in \partial\Omega\},
	\end{equation*}
	which was used as admissible set of deformations in \cite{BER15}.
	By the results of Ball \cite{Ball1981}, we know that $\varphi(\bar{\Omega}) = \bar{\Omega}$ nad $\varphi$ is a.e. injective. 
	By the last property and since $\overline \Omega$ is  bounded, it follows immediately  
	for all $\varphi \in \diffeo$ that
	\begin{equation} \label{sieben}
	\| \varphi \|_{(L^\infty(\Omega))^n} \le C, \qquad \| \varphi \|_{(L^2(\Omega))^n} \le  C,
	\end{equation}
	with constants depending only on $\Omega$. We have $\varphi \in C^{1,\alpha}(\overline \Omega)$ 
		and by the inverse mapping theorem  $\varphi^{-1}$ exits locally around a.e.~$x \in \Omega$ and is continuously differentiable on the corresponding neighborhood.
	However, to guarantee that $\varphi^{-1}$ is continuous (or, even more, continuously differentiable) further assumptions
	are required, see \cite[Theorem 2]{Ball1981}.
	Take for example the function $\varphi(x) \coloneqq x^3$ on $\Omega \coloneqq (-1,1)$  which is in $\diffeo$,  
	but $\varphi^{-1} (x) = \mathrm{sgn} (x) |x|^\frac13$
	is not continuously differentiable. 
	Furthermore,  $I \in L^2(\Omega,\HH)$ and $\varphi \in \diffeo$ do not guarantee that $I \circ \varphi \in L^2(\Omega,\HH)$,
	as the example $I(x) \coloneq x^{-\frac14}$ in $L^2((0,1),\mathbb R)$ and $\varphi(x) \coloneqq x^2$ shows. 
		
	Therefore, we introduce, for \emph{small} fixed $\epsilon > 0$, the \emph{admissible set}
	\begin{equation*}\label{admiss_eps}
	\mathcal{A}_\epsilon \coloneqq \{\varphi\in \left(  W^{m,2}(\Omega) \right)^n: \det(D\varphi) \ge \epsilon, \; \varphi(x) = x  \; \mathrm{for} \; x \in \partial\Omega\} \subset \diffeo.
	\end{equation*}
	Later, in Theorem \ref{lem:uni:seq}, we have to solve a system of equations which entries depend on $\det(D\varphi)$. 
	For stability reasons, we do not want that the determinant becomes arbitrary small. 
	This can be avoided by introducing $\epsilon$.
	Moreover, by the inverse mapping theorem, $\varphi\in\mathcal{A}_\epsilon$ is a diffeomorphism, although in general $\varphi^{-1} \notin\mathcal{A}_\epsilon$.
	Further,  $I \in L^2(\Omega,\HH)$ and $\varphi \in \mathcal{A}_\epsilon$ imply $I \circ \varphi \in L^2(\Omega,\HH)$.
	We mention that the space of images $L^\infty(\Omega, \mathbb R)$ was discussed in the thesis \cite{Effland17}.
	However, working in the Hadamard space $L^2(\Omega,\HH)$ simplifies the proofs, in particular we can use
	the concept of weak convergence in these spaces.
	

\section{Minimizers of the Space Continuous Model}\label{sec:dm}
	Let $\HH$ be a finite dimensional Hadamard manifold.
	Due to our application we call the functions from $L^2(\Omega,\HH)$ images.
	Mappings $\varphi \in \mathcal{A}_\epsilon$ 
	can act on images $I \in L^2(\Omega,\HH)$ 
	by 
	\[\varphi\circ I = I(\varphi), \quad \varphi \in \mathcal{A}_\epsilon .\]
	
	Inspired by the time discrete geodesic paths model for images with values in $\mathbb{R}^n$ proposed in \cite{BER15},  
	we introduce a general morphing model for manifold-valued images and prove the existence of minimizers of this model
	in the next subsection. 
	Although we can basically follow the lines in \cite{BER15},
	all proofs are new even for the Euclidean setting due to the different admissible set.
	Moreover, the manifold-valued setting requires some care when considering the minimization of the image sequence.
	Then, in Subsection \ref{sec:special_model}, we specify the model by choosing the linearized elastic potential as regularizer.
	
	\subsection{Space Continuous Model}\label{sec:general_model}
	
	Let $W \colon\RR^{n,n}\to\RR_{\le 0}$ be a lsc  mapping. 	
	Let $\gamma>0$ and $m>1+\frac{n}{2}$. Let $K \ge 2$ be an integer.
	Given a template image and a reference image
	\[I_0 = T \in  L^2(\Omega,\HH), \quad I_K = R \in  L^2(\Omega,\HH),\]
	respectively, we are searching for an image sequence 
	\[\vec I \coloneqq (I_1,\dots,I_{K-1}) \in \left( L^2(\Omega,\HH) \right)^{K-1},\] see Fig.~\ref{fig:sequence},
	which minimizes the energy
	\begin{align} \label{d_path}
	\boldsymbol{\mathcal{J}} (\vec I) 
	&\coloneqq \sum_{k = 1}^{K}  \inf_{\varphi_k\in \mathcal{A}_\epsilon} 
	\int_\Omega W(D\varphi_k(x))+\gamma\lvert D^m\varphi_k(x)\rvert^2\dx x
	+ \mathrm{d}_2^2(I_{k-1}\circ \varphi_{k},I_{k})\\\notag
	&= \inf_{\boldsymbol{\varphi} \in \mathcal{A}_\epsilon^K}  \sum_{k = 1}^{K}\int_\Omega W(D\varphi_k(x))+\gamma\lvert D^m\varphi_k(x)\rvert^2\dx x
	+ \mathrm{d}_2^2(I_{k-1}\circ \varphi_{k},I_{k})\\\notag
	&= \inf_{\boldsymbol{\varphi} \in \mathcal{A}_\epsilon^K} \mathcal{J}(\vec I, \boldsymbol{\varphi}),
	\end{align}
	where $\boldsymbol{\varphi} \coloneqq (\varphi_1,\dots,\varphi_K)$
	and
	\begin{align*}
	\mathcal{J}(\vec I, \boldsymbol{\varphi}) 
	&\coloneqq
	\sum_{k = 1}^{K} 
	\int_\Omega W(D\varphi_k(x))+\gamma\lvert D^m\varphi_k(x)\rvert^2\dx x
	+ \mathrm{d}_2^2(I_{k-1}\circ \varphi_{k},I_{k}) .
	\end{align*}
	Note that for simplicity of the notation we moved the parameter $\frac{1}{\delta}$ of the squared distance term to $W$, see \eqref{eq_spez_W},
	and used a shift in the $\varphi_k$.
	
	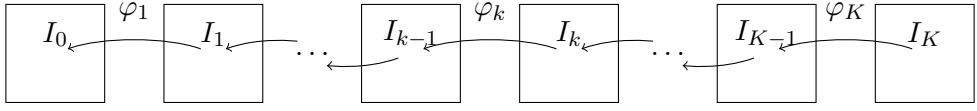
\begin{figure}[t]
	\centering
	\begin{TikZOrPDF}{ImageSeq}
		\begin{tikzpicture}[scale=1.3]	
		\draw (0,0) rectangle (1,1);
		\node at (0.5,.7) (I0) {$I_0$};
		\draw[<-,shorten <= 5,shorten >= 5] (.5,.5) ..controls +(20:.5)  and +(160:.5).. (2.1,.5) node [pos = 0.5,label = 90:$\varphi_1$](phi0) {};
		\draw (1.6,0) rectangle (2.6,1);
		\node at (2.1,.7) (I1) {$I_1$};
		\draw[<-,shorten <= 5] (2.1,.5) ..controls +(20:.5)  and +(180:.025).. (2.95,.625);
		
		\node at (3.1,.5) (d1) {$\Huge{\dots}$};
		\draw[<-,shorten >= 5] (3.25,.375) ..controls +(0:.025)  and +(200:.5).. (4.1,.5);
		\draw (3.6,0) rectangle (4.6,1);
		\node at (4.1,.7) (Ik) {$I_{k-1}$};
		\draw[<-,shorten <= 5,shorten >= 5] (4.1,.5) ..controls +(20:.5)  and +(160:.5).. (5.7,.5) node [pos = 0.5,label = 90:$\varphi_{k}$](phik) {};
		\draw (5.2,0) rectangle (6.2,1);
		\node at (5.7,.7) (Ik1) {$I_{k}$};
		\draw[<-,shorten <= 5] (5.7,.5) ..controls +(20:.5)  and +(180:.025).. (6.55,.625);
		\node at (6.7,.5) (Ik) {$\Huge{\dots}$};
		\draw[<-,shorten >= 5] (6.85,.375) ..controls +(0:.025)  and +(200:.5).. (7.7,.5);
		\draw (7.2,0) rectangle (8.2,1);
		\node at (7.7,.7) (Ik) {$I_{K-1}$};
		\draw[<-,shorten <= 5,shorten >= 5] (7.7,.5) ..controls +(20:.5)  and +(160:.5).. (9.3,.5) node [pos = 0.5,label = 90:$\varphi_{K}$](phiK) {};
		\draw (8.8,0) rectangle (9.8,1);
		\node at (9.3,.7) (Ik1) {$I_{K}$};
		\end{tikzpicture}
	\end{TikZOrPDF}
	\caption{Illustration of the image and the diffeomorphism path\label{fig:sequence}}
\end{figure}
	
	Typical strategies for minimizing such a functional is alternating minimization over $\vec I$ and $\boldsymbol{\varphi}$.
	In the following, we show that the corresponding subproblems have a minimizer.
	The results can then be used to show that the whole functional has a minimizer.
	\\
	
	First, we fix the image sequence $\vec{I} \in \left( L^2(\Omega,\HH) \right)^{K-1}$ and  show that $\mathcal{J}(\vec I, \cdot)$ 
	has a minimizer $\boldsymbol{\varphi} \in \diffeo^K$. 
	It suffices to prove that each of 
	the registration problems 
	\begin{equation*}
	{\mathcal R} (\varphi_k;I_{k-1},I_k) \coloneqq \int_\Omega W(D\varphi_k(x)) + \gamma\lvert D^m\varphi_k(x)\rvert^2\dx x
	+ \mathrm{d}_2^2(I_{k-1}\circ \varphi_{k},I_{k})
	\end{equation*}
	has a minimizer in $\mathcal{A}_\epsilon$, $k=1,\ldots,K$. Note that we can show the more general result for
	$\varphi \in \diffeo$ for this functional. If we restrict ourselves to  $\varphi \in \mathcal{A}_\epsilon$  the proof can be simplified,
	see Corollary~\ref{lem:ex:phi_eps}.

\begin{theorem}\label{lem:ex:phi}
		Let $W\colon\RR^{n, n}\to\RR_{\le 0}$ be a lsc mapping with the property
	\begin{equation} \label{prop3}
	W(A)=\infty \quad \mathrm{if} \quad \det A \le 0.
	\end{equation}
		Further, let $T,R\in L^2(\Omega,\HH)$ be given.
		Then there exists  $\hat \varphi\in\diffeo$ minimizing
		\begin{equation*}
		 \mathcal{R}(\varphi;T,R)  =
		\int_{\Omega}W(D \varphi)+\gamma\lvert D^m \varphi\rvert^2 + d^2\bigr(T\circ\varphi(x),R(x)\bigl) \dx x
		\end{equation*}
		over all $\varphi \in \diffeo$.
\end{theorem}

\begin{proof} 
	1. Let $\{ \varphi^{(j)} \}_{j\in\NN}$, $\varphi^{(j)} \in \diffeo,$ be a minimizing sequence
	of $\mathcal{R}$.  
	Then  we have $\mathcal{R}( \varphi^{(j)};T,R ) \le C$ for all $j \in \mathbb N$.
	This implies that $\{ \varphi^{(j)}\}_{j\in\NN}$ has uniformly bounded $(W^{m,2}(\Omega))^n$ semi-norm,
	and by \eqref{sieben} the sequence is also uniformly bounded in $(L^2(\Omega))^n$.	
	Now we apply the Gagliardo-Nirenberg inequality, see Appendix~\ref{app:gn} Theorem~\ref{th:gn} in the form given in Remark~\ref{th:gnr}, which states that for all $0 \le i < m$ it holds
	\begin{equation*}
	\lVert D^i \varphi^{(j)}\rVert_{L^2(\Omega)}
	\le C_1\lVert D^m \varphi_\nu^{(j)}\rVert_{L^2(\Omega)}
	+C_2\lVert \varphi_\nu^{(j)}\rVert_{L^2(\Omega)}, \quad \nu =1,\ldots,n.
	\end{equation*}%
	All terms on the right hand side are uniformly bounded.
	Hence, the $(W^{m,2}(\Omega))^n$ norm of $\{ \varphi^{(j)}\}_{j\in\NN}$ is uniformly bounded.
	Since $W^{m,2}(\Omega)$ is reflexive, there exists a subsequence $\{ \varphi^{(j_l)} \}_{l\in\NN}$, 
	which converges weakly to some function $\hat \varphi$ in $\left( W^{m,2} (\Omega)\right)^n$. 
	By the compact embedding $W^{m,2} (\Omega)\hookrightarrow C^{1,\alpha}\big(\overline{\Omega}\big)$, $\alpha \in (0,m-1-\frac{n}{2})$,
	this subsequence converges strongly to $\hat \varphi$ in $\left( C^{1,\alpha}\big(\overline{\Omega}\big) \right)^n$.
	We denote this subsequence by $\{ \varphi^{(j)} \}_{j\in\NN}$ again.
	\\[1ex]
	2. Next we show that $\hat \varphi$ is in the set ${\mathcal A}$.
	By part 1 of the proof $D \varphi^{(j)}$ converges uniformly to $D \hat \varphi$.
	Since $W$ is lsc, this implies 
	\[
	\liminf_{j\to\infty} W(D \varphi^{(j)})(x) \ge W(D \hat \varphi)(x)
	\]
	for all $x \in \Omega$ and since $W$ is nonnegative we obtain by Fatou's lemma 
	\[
	\int_{\Omega} W(D \hat \varphi(x))\dx x \le \liminf_{j \rightarrow \infty} \int_{\Omega} W(D\varphi^{(j)} (x))\dx x \le C.
	\]
	By \eqref{prop3}  this implies $\det( D \hat \varphi(x)) > 0$ a.e.. Further the boundary condition is fulfilled so that $\hat \varphi\in\diffeo$. 
	It remains to show that $\lim_{j\to\infty}\mathcal{R}(\varphi^{(j)};T,R) = \mathcal{R}(\hat \varphi;T,R)$. 
	\\[1ex]
	3. We prove that 
	\[\mathrm{d}^2_2(T( \hat \varphi),R) \leq \liminf_{j \to \infty} \mathrm{d}^2_2(T(\varphi^{(j)}),R).\]		
	Assume $\mathrm{d}^2_2(T( \hat \varphi),R) > \liminf_{j \to \infty} \mathrm{d}_2^2(T(\varphi^{(j)}),R)$. 
	Further, assume for a moment that $\mathrm{d}^2_2(T( \hat \varphi),R)$ is finite.
	Then we can find an $\delta>0$ such that 
	$$\mathrm{d}_2^2(T( \hat \varphi),R) - \delta= \liminf_{j \to \infty} \mathrm{d}_2^2(T(\varphi^{(j)}),R).$$ 
	The sets 
	$\{x \in \Omega: \det \varphi^{(j)} = 0, \;  j \in \mathbb N \}$
	and $\{x \in \Omega: \det \hat \varphi = 0\}$ have measure zero.
	Let  $\widetilde{\Omega}$ be their complementary set, which is open since the $\varphi^{(j)}$ are convergent.
	By the inverse mapping theorem we know that $(\varphi^{(j)})^{-1}$ and $\hat \varphi^{-1}$ 
	are continuously differentiable on $\widetilde{\Omega}$.
	Since $\widetilde \Omega$ is open we can use monotone convergence 
	to find a compact set $\mathcal K$ such that ${\mathcal K} \subset \widetilde \Omega$ and
	\[\mathrm{d}^2_2(T( \hat \varphi),R) \leq \int_{{\mathcal K}}d^2\bigl(T(\hat \varphi(x)),R(x)\bigr) \dx x + \frac{\delta}{2}.\] 
	Using that the integrands are nonnegative, we get
	\begin{equation} \label{star}
	\liminf_{j \to \infty}\int_{{\mathcal K}}d^2\bigl(T(\varphi^{(j)}(x)),R(x)\bigr) \dx x < \int_{{\mathcal K}}d^2\bigl(T(\hat \varphi(x)),R(x)\bigr) \dx x.
	\end{equation}
	For the case $\mathrm{d}^2_2(T( \hat \varphi),R)  = \infty$ we obtain the same inequality, since the left-hand side is less than some finite constant $C$ independent of $\mathcal K$. However, $\mathcal K$ can be chosen such that the right-hand side gets arbitrary large using monotone convergence. Note that it will be still finite due to the change of variables formula, \cite[Theorem 7.26]{Rudin1964}.

	Now, for the set $\mathcal K$, we obtain
		\begin{align}
		\notag&\bigg\vert\int_{{\mathcal K}}d^2\bigl(T(\varphi^{(j)}(x)),R(x)\bigr) - d^2\bigl(T(\hat \varphi(x)),R(x)\bigr)\dx x \bigg \vert\\\notag 
		\le&
		\int_{{\mathcal K}}\Bigl(d\bigl(T(\varphi^{(j)}(x)),R(x)\bigr)+
		d\bigl(T(\hat \varphi(x)),R(x)\bigr)\Bigr)\\\notag& \quad\Bigl\vert d\bigl(T(\varphi^{(j)}(x)),R(x)\bigr)-
		d\bigl(T(\hat \varphi(x)),R(x)\bigr)\Bigr \vert \dx x\\\notag
		\le&
		\int_{{\mathcal K}}\Bigl(d\bigl(T(\varphi^{(j)}(x)),R(x)\bigr)+d\bigl(T(\hat \varphi(x)),R(x)\bigr)\Bigr)d\Bigl(T\bigl(\varphi^{(j)}(x)\bigr),T\bigl(\hat \varphi(x)\bigr)\Bigr)\dx x\\
		\le& \, C_1\sqrt{\int_{{\mathcal K}}d^2\Bigl(T\bigl(\varphi^{(j)}(x)\bigr),T\bigl(\hat \varphi(x)\bigr)\Bigr)\dx x} \label{eq:ineq:data}
		\end{align}
		with
		\[C_1 \coloneqq \sqrt{\int_{{\mathcal K}} d^2\bigl(T(\varphi^{(j)}(x)),R(x)\bigr)\dx x} + \sqrt{\int_{{\mathcal K}} d^2\bigl(T(\hat \varphi(x)),R(x)\bigr)\dx x} < \infty.\]
		This constant is finite since the first term is uniformly bounded by construction and the second term is bounded by construction of $\mathcal K$. Now, Corollary~\ref{stet_norm} implies that the term
		\begin{align*}
		\int_{{\mathcal K}} d^2\Bigl(T\bigl(\varphi^{(j)}(x)\bigr),T\bigl(\hat \varphi(x)\bigr)\Bigr)\dx x 
		\end{align*}
		in equation~\eqref{eq:ineq:data} converges to zero. This yields a contradiction to \eqref{star}.\\		
		4. 
		By the previous steps we have that  the three summands in $\mathcal{R}$ are (weakly) lower semicontinuous.
		Then we get
		\begin{align*}
		\mathcal{R}(\hat \varphi;T,R) 
		&\le \liminf_{j\to\infty}  \int_{\Omega}W(D \varphi^{(j)})+\gamma\lvert D^m \varphi^{(j)} \rvert^2+d^2\bigr(T(\varphi^{(j)} (x)),R(x)\bigl) \dx x\\
		&= \inf_{\varphi\in\diffeo}\mathcal{R}( \varphi)
		\end{align*}
		which proves the claim.
\end{proof}

Following the lines of the previous theorem, we can prove its analogue for $\mathcal{A}_\epsilon$.

 \begin{corollary}\label{lem:ex:phi_eps}
		Let $W\colon\RR^{n, n}\to\RR_+$ be a lsc mapping.
		Further, let $T,R\in L^2(\Omega,\HH)$ be given.
		Then there exists  $\hat \varphi\in\mathcal{A}_\epsilon$ minimizing
		\begin{equation*}
		 \mathcal{R}(\varphi;T,R) \coloneqq
		\int_{\Omega}W(D \varphi)+\gamma\lvert D^m \varphi\rvert^2 + d^2\bigr(T\circ\varphi(x),R(x)\bigl) \dx x
		\end{equation*}
		over all $\varphi \in \mathcal{A}_\epsilon$.
\end{corollary}
	
\begin{proof}
The proof follows the lines of the previous one, but simplifies since all integrals are defined over $\Omega$ and the approximation with compact sets is not necessary.
\end{proof}

Next we fix the sequence of  admissible mappings $\boldsymbol{\varphi}\in \mathcal{A}_\epsilon^K$ and ask for a minimizer of 
\begin{align*}
J_{\boldsymbol{\varphi}}(\vec I) 
\coloneqq 
\sum_{k = 1}^{K}  \mathrm{d}^2_2(I_{k-1}\circ \varphi_{k},I_{k}) \quad \mbox{subject to} \quad I_0 = T, \; I_K = R.
\end{align*}
Note that we consider admissible functions in $\mathcal{A}_\epsilon$ now,
since the composition $I \circ \varphi$ of a deformation $\varphi$ and an image $I\in L^2(\Omega,\HH)$ 
should be again in $L^2(\Omega,\HH)$.
	
\begin{lemma}\label{lem:ex:seq}
For fixed $\boldsymbol{\varphi} \in \mathcal{A}_\epsilon^K$,
there exists a sequence $\hat{\vec{I}} \in \left( L^2(\Omega,\HH) \right)^{K-1}$ 
minimizing $J_{\boldsymbol{\varphi}}$.
\end{lemma}
	
	\begin{proof}	
		Let $\{ {\vec I}^{(j)} \}_{j\in\mathbb{N}}$ be a minimizing sequence of $J_{\boldsymbol{\varphi}}$,
		i.e., 
		$ \lim_{j\rightarrow \infty} J_{\boldsymbol{\varphi} } \bigl(\vec I^{(j)}\bigr)= \hat J$, 
		where $\hat J$ is the infimum of $J_{\boldsymbol{\varphi}}$.
		Clearly, there exists $C \ge 0$ such that $J_{\boldsymbol{\varphi}} ({\vec I}^{(j)}) \le C$ for all $j \in \mathbb N$.
		By the triangle inequality we obtain for some fixed $a\in\HH$ that
		\begin{equation*}
		\mathrm{d}_2(I_{k+1}^{(j)},a) 
		\le 
		\mathrm{d}_2(I_k^{(j)}\circ\varphi_{k+1},I_{k+1}^{(j)})+d_2(I_{k}^{(j)}\circ\varphi_{k+1},a) \le C +\mathrm{d}_2(I_{k}^{(j)}\circ\varphi_{k+1},a).
		\end{equation*}
		Thus, since $\det (D \varphi_k) \ge \epsilon $ a.e.~on $\Omega$ for $k=1,\ldots,K$,
		\begin{align*}
		 \mathrm{d}_2(I_1^{(j)},a) 
		 & \le C + \mathrm{d}_2(T \circ \varphi_1,a)\\
		 \mathrm{d}_2(I_2^{(j)},a) 
		 & \le C + \mathrm{d}_2( I_1^{(j)} \circ \varphi_2,a) \\
		 &= C + \left( \int_{\Omega} d^2 \left( I_1^{(j)} \left( \varphi_2(x) \right),a \right) \, \dx x \right)^\frac12\\
		 &= C + \left( \int_{\Omega} d^2 \left( I_1^{(j)} \left( x \right),a \right) \, \big\lvert\det \big(D \varphi_2(\varphi_2^{-1}(x))\big) \big\rvert^{-1} \dx x \right)^\frac12\\
		 &\le C +   \epsilon^{-\frac12} \mathrm{d}_2 (I_1^{(j)},a) \le C +   \epsilon^{-\frac12} \big(C + \mathrm{d}_2(T \circ \varphi_1,a)\big).
		\end{align*}
		Continuing this successively we see that 		
		$\{ \vec I^{(j)}\}_{j\in\mathbb{N}}$
		is bounded in $\left( L^2(\Omega,\HH) \right)^{K-1}$.
		From \cite[Proposition 3.1.2]{Bac14} we know that a bounded sequence in an Hadamard space 
		has a weakly convergent subsequence $\{ \vec I^{(j_k)}\}_{k\in\mathbb{N}}$.
		Let  $\hat{ \vec I} \in \left( L^2(\Omega,\HH) \right)^{K-1}$ be its weak limit point.
		Now $\mathrm{d}_2(\cdot,\cdot)$ is a continuous convex function and the same holds true for $J_{\boldsymbol{\varphi}}$.
		Then, by \cite[Lemma 3.2.3]{Bac14}, the function $J_{\boldsymbol{\varphi}}$ is weakly lsc
		which means that
		$$
		\hat J = \lim_{j\rightarrow \infty} J_{\boldsymbol{\varphi} } \bigl(\vec I^{(j)}\bigr)
		=
		\lim_{k \rightarrow \infty} J_{\boldsymbol{\varphi} }
				\bigl( \vec I^{(j_k)} \bigr) 
				\ge
		J_{\boldsymbol{\varphi} }\bigl(\hat {\vec I}\bigr),
		$$
		so that 
		$\hat{\vec I}$ is a minimizer of the functional.
		\end{proof}

\begin{theorem}\label{lem:uni:seq}
	For fixed $\boldsymbol{\varphi} \in \mathcal{A}_\epsilon^K$, there exists a 
	unique  sequence of intermediate images $\vec{I} \in \bigl(L^2(\Omega,\HH)\bigr)^{K-1}$ 
	minimizing $J_{\boldsymbol{\varphi}}$. 
\end{theorem}	
		
		\begin{proof} 
		By Lemma \ref{lem:ex:seq}, there exists a minimizer of $J_{\boldsymbol{\varphi}}$.
		Setting 
\begin{align}
 \psi_K(x) &\coloneqq x,\notag\\
 \psi_{k}(x) &\coloneqq \varphi_{k+1} \circ \psi_{k+1} (x) = \varphi_{k+1} \circ \ldots \circ \varphi_{K}(x), \quad k=K-1,\ldots,0,
\label{def_psi}
\end{align}
and substituting $x \coloneqq \psi_k (y)$
in the $k$-th summand of $J_{\boldsymbol{\varphi}}$, we obtain
\begin{align*}
J_{\boldsymbol{\varphi}} (\vec I)
&= 
\sum_{k=1}^K 
\int_{\Omega} 
d ^2 \left( I_k \left( \psi_{k}(y) \right) , I_{k-1} \left( \varphi_k  \circ \psi_{k} (y)  \right) \right)\, 
|\det \left( D \psi_k(y) \right)| \, \dx  y \\
&= 
\sum_{k=1}^K \int_{\Omega} 
d^2 \left( I_k\left( \psi_{k}(x) \right) , I_{k-1} \left(\psi_{k-1} (x)  \right) \right)\, 
|\det \left( D \psi_k(x) \right)| \, \dx  x .
\end{align*}
Using $F_K \coloneqq I_K = R$,
$F_k \coloneqq I_k \circ \psi_{k}$, 
and
$w_k(x) \coloneqq  |\det \left( D \psi_k(x) \right) | > 0$, 
we are concerned with the minimization of
\begin{equation*}
\sum_{k=1}^K \int_{\Omega}  w_k(x) d^2\left( F_k (x),F_{k-1} (x) \right)  \, \dx  x
\quad \mbox{subject to} \quad F_0  = T \circ \psi_0, \; F_K = R.
\end{equation*}
	Each image $I_k$, resp. $F_k$  appears only in two summands 
		\begin{equation*}
		\int_{\Omega} w_k(x) d^2 \left( F_k (x),F_{k-1} (x) \right) + 
		w_{k+1}(x) d^2 \left (F_{k+1} ,F_k(x) \right)\dx x.
		\end{equation*}
		By the Euler-Lagrange equation and since the Riemannian gradient of the squared distance function is
		$\nabla d^2(\cdot,a) (b) = - 2 \log_b a$, we obtain 
		\begin{equation} \label{system}
		w_k \log_{ F_k} F_{k-1} + w_{k+1} \log _{F_k} F_{k+1} = 0, \quad k=1,\ldots,K-1.
		\end{equation}
		If $K=2$, we have only one equation
		\begin{equation*}
		w_1\log_{F_1}F_0+ w_2 \log_{F_1}F_2 = 0,
		\end{equation*}
		which implies that $F_1$ has to be on the geodesic $\gamma_{\overset{\frown}{F_0,F_2}}(t),t\in[0,1]$. 
		Since 
		$$
		\frac{|\log_{F_1} F_0|}{|\log_{F_1} F_2|} = \frac{w_2}{w_1}
		$$
		we obtain
		$$F_1 = \gamma_{\overset{\frown}{F_0,F_2}}\biggl(\frac{w_2}{w_1 + w_2}\biggr).$$
		For general $K \ge 3$, the system of equations \eqref{system} can only be fulfilled
		if three consecutive points always lie on a geodesic. This
		is only possible if all points are on the same geodesic $\gamma_{\overset{\frown}{F_0,F_K}}(t), t\in[0,1]$.
		More precisely, 
		\begin{equation} \label{getFF}
		F_k = \gamma_{\overset{\frown} {F_0,F_K}}(t_k), \quad k=1,\ldots,K-1,
		\end{equation}
		where by \eqref{system},
		the $t_k$ are related by
		\begin{equation}\label{eq:samp}
		\frac{s_k}{s_{k+1}} = \frac{w_{k+1}}{w_{k}}, \quad  s_k \coloneqq t_{k}-t_{k-1}, \quad k=1,\ldots,K-1,
		\end{equation}
		\begin{figure}
			\centering
			\includegraphics{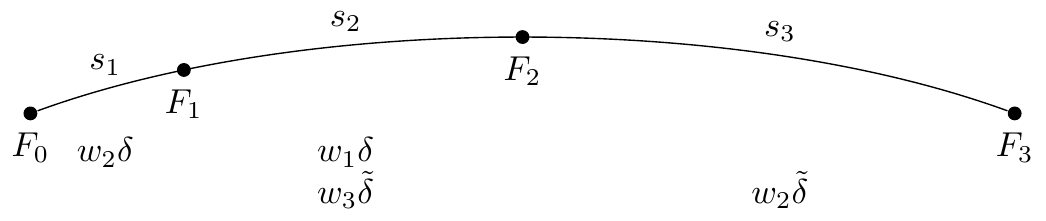}
			\caption{Illustration of relation~\eqref{eq:samp} for the geodesic~\eqref{getFF} with $K = 3$, where $\delta$, $\tilde{\delta}$ are constants canceling out in the fraction.}\label{path_rel}
		\end{figure}
		see Fig.~\ref{path_rel} for an illustration.
		It is easy to check that these conditions are fulfilled by
		$$
		s_k \coloneqq \frac{\alpha_k}{\sum_{k=1}^{K-1} \alpha_k} , \quad 
		\alpha_k \coloneqq  	\prod_{i=1 \atop {i\not = k}}^K w_i, \quad k=1,\ldots,K-1,
		$$
		so that
		\begin{equation*}
		t_k = \sum_{i=1}^k s_i = \frac{\sum_{i=1}^k w_i^{-1}}{\sum_{i=1}^K w_i^{-1}}.
		\end{equation*}
		The geodesics between $F_0$ and $F_K$ are unique, so that the resulting points $F_k$, $k =1,\dots,K-1$, are unique as well.
		As we know that $\varphi_k$ 
		are diffeomorphisms, the functions $\psi_k$, $k = 1,\dots, K-1,$ are diffeomorphisms as well.
		By \eqref{getFF}, convexity of $d^2$, and $F_0,R \in L^2(\Omega,\HH)$, we see that
		$F_k\in L^2(\Omega,\HH)$ and thus $I_k\in L^2(\Omega,\HH)$, $k=1,\ldots, K-1$.
	\end{proof}
	
	To prove the next theorem we need the following corollary.
	
	\begin{corollary}\label{ConvergenceFK}
		Let $\{ \boldsymbol{\varphi}^{(j)} \}_{j \in \mathbb{N}}$ with $\boldsymbol{\varphi}^{(j)} \in \mathcal{A}_\epsilon^{K}$
		be a sequence which converges in  $(C^{1,\alpha}(\overline \Omega) )^{nK}\!\!$.
		For each $j \in \mathbb{N}$, let $\vec{I}^{(j)} \in \left( L^2(\Omega,\HH) \right)^{K-1}$ be the minimizer of
		$J_{\boldsymbol{\varphi}^{(j)}}$. 
		Then the sequence $\{ \vec{I}^{(j)} \}_{j\in \mathbb{N}}$ converges in $\left( L^2(\Omega,\HH) \right)^{K-1}$.		
	\end{corollary}
	
	\begin{proof} The proof follows the path of the previous theorem.
		Similarly as in \eqref{def_psi}, we define 
		\begin{align*}
		\psi_K(x) \coloneqq x, \quad
		\psi_k^{(j)} \coloneqq \varphi_{k+1}^{(j)} \circ \psi_{k+1}^{(j)}, \quad
		F_k^{(j)} \coloneqq I_k^{(j)} \circ \psi_k^{(j)}, \quad k=0,\ldots,K-1,
		\end{align*}
		where $I_0^{(j)} = T$ and $F_K^{(j)} = R$.
		Clearly, the convergence of $\{ \boldsymbol{\varphi}^{(j)} \}_{j\in\NN}$ in $\left(C^{1,\alpha}(\overline{\Omega})\right)^n$
		implies the convergence of $\{\psi^{(j)}_k\}_{j\in\NN}$ to $\hat \psi_k$ in $\left(C^{1,\alpha}(\overline{\Omega})\right)^n$
		for all $k=0,\ldots,K-1$.
		Hence, for all  $k=0,\ldots,K-1$,
		\[ w_k^{(j)} (x)  \coloneqq \bigl\lvert\det (D\psi_k^{(j)}(x)\bigr)\bigr\rvert\]
		is uniformly convergent on $\overline{\Omega}$. By construction, the $w_k^{(j)} \ge \tilde \epsilon > 0$ such that
		$$
		t_k ^{(j)} (x) \coloneqq \frac{\sum_{i=1}^k \left(w_i^{(j)} \right)^{-1}}{\sum_{i=1}^K \left(w_i^{(j)} \right)^{-1}}
		$$
		converges pointwise on $\overline{\Omega}$ as $j\rightarrow \infty$.
		We denote the limit by $\hat t_k$ and set
		\begin{align*}
		\hat F_0(x) &\coloneqq (T \circ \hat \psi_0)(x),\\
		\hat F_k(x) &\coloneqq \gamma_{\overset{\frown}{\hat F_0(x),R(x)}} \left(\hat t_k(x)\right), \quad k=1,\ldots,K-1.
		\end{align*}
		Recall that as in the proof of the previous theorem
		$F_k^{(j)}(x) = \gamma_{\overset{\frown}{\hat F_0^{(j)} (x),R(x)}} \left(t_k^{(j)}  (x) \right)$.
		Using Corollary \ref{stet_norm}, we see that 
		$F_0^{(j)} = T \circ \psi_0^{(j)}$ converges in 
		$L^2(\Omega,\HH)$ to $\hat F_0 = T \circ \hat \psi_0$.
		For $k=1,\dots,K-1$, we obtain
			\begin{align*}
			\mathrm{d}_2(F_k^{(j)}, \hat F_k) 
			\le& 
			\;  \mathrm{d}_2 \Big( F_k^{(j)}, \gamma_{\overset{\frown}{\hat F_0(x),R(x)}} \left( t_k^{(j)} (x) \right) \Big)
			+   \mathrm{d}_2 \Big( \gamma_{\overset{\frown}{\hat F_0(x),R(x)}} \left( t_k^{(j)}  (x) \right), \hat F_k \Big)
			\\
			=& \; 
			\left( 
			\int_{\Omega} d^2 
			\Big( 
			\gamma_{\overset{\frown}{F_0^{(j)}(x),R(x)}} \left( t_k^{(j)} (x)  \right), 
			\gamma_{\overset{\frown}{\hat F_0(x),R(x)}} \left(t_k^{(j)}  (x) \right) 
			\Big )\dx x
			\right)^\frac12
			 \\
			&+
			\left( \int_{\Omega} d^2
			\Big(\gamma_{\overset{\frown}{\hat F_0(x),R(x)}} \left(t_k^{(j)}  (x) \right), 
			\gamma_{\overset{\frown}{\hat F_0(x),R(x)}} \left(\hat t_k  (x) \right)
			\Big) \dx x\right)^\frac12. 
			\end{align*}
		Since the distance $d$ is jointly convex, we estimate	by \eqref{eq:jointconvexity},
		$$
		d^2 
			\Big( 
			\gamma_{\overset{\frown}{\hat F_0^{(j)}(x),R(x)}} \left( t_k^{(j)} (x)  \right), 
			\gamma_{\overset{\frown}{\hat F_0(x),R(x)}} \left(t_k^{(j)}  (x) \right) 
			\Big ) 
			\le
			\left( 1-  t_k^{(j)}(x)\right)^2  d^2 \left( F_0^{(j)}(x),\hat F_0(x) \right). 
		$$
		Using this inequality, property \eqref{eq:geo} of the geodesic and that by definition $0 < w_k^{(j)} < 1$ we conclude
		\begin{align*}
		\mathrm{d}_2 (F_k^{(j)}, \hat F_k) 
			\le& 
			\left( \int_{\Omega} 
			d^2 \left( F_0^{(j)}(x),\hat F_0(x) \right) 
			\left( 1-  t_k^{(j)}(x)\right)^2  \dx x
			\right)^\frac12\\
			& + 
			\left( \int_{\Omega} 
			d^2 \left(\hat F_0(x), R(x) \right) 
			\lvert  t_k^{(j)}(x) - \hat t_k(x) \rvert^2  \dx x
			\right)^\frac12\\
			&\le 
			\left( \int_{\Omega} 
			d^2 \bigl(F_0^{(j)}(x),\hat F_0(x) \bigr) \dx x\right)^\frac12 \\
			&+  
			\left( \int_{\Omega} 
			d^2 \left(\hat F_0(x), R(x) \right) 
			\lvert  t_k^{(j)}(x) - \hat t_k(x) \rvert^2  \dx x
			\right)^\frac12.
		\end{align*}
			Since the second factor in the last integral converges pointwise to zero, as 
			$w^{(j)}_k\to\hat w_k$ in $C^{\alpha}(\overline{\Omega})$, and 
			$d^2 \bigl(\hat F_0(x), R(x)\bigr)\lvert  t_k^{(j)}(x) - \hat t_k(x) \rvert^2$ 
			has the integrable upper bound
			$4d^2 \bigl(\hat F_0(x), R(x)\bigr)$, we conclude by the convergence theorem of Lebesgue that $F_k^{(j)}$ converges in $L^2(\Omega,\HH)$ to $\hat F_k$, $k=1,\ldots,K-1$.
						
			Finally, setting $\hat I_k \coloneqq \hat F_k \circ (\hat \psi_k)^{-1}$, $k=1,\ldots,K-1$, we obtain
			\begin{align*}
			 \mathrm{d}_2 \big( I^{(j)}_k, \hat I_k \big)
			 \le& \,
			 \mathrm{d}_2 \big( I^{(j)}_k, \hat F_k \circ  (\psi^{(j)}_k )^{-1} \big) 
			 +
			  \mathrm{d}_2 \big( \hat F_k \circ  (\psi^{(j)}_k )^{-1}, \hat F_k \circ (\hat \psi_k)^{-1} \big)
				\end{align*}
			and since 	$\{\psi^{(j)}_k\}_j{j\in\NN}$ converges in $\left(C^{1,\alpha}(\overline{\Omega})\right)^n$ further
			\begin{align*}
			 \mathrm{d}_2 \big( I^{(j)}_k, \hat I_k \big)
			   \le& \,C(\epsilon) \Bigl(
			  \mathrm{d}_2 \big( F_k^{(j)},\hat F_k) 
			  +  \mathrm{d}_2 \big( \hat F_k, \hat F_k \circ (\hat \psi_k)^{-1}\circ  (\psi^{(j)}_k ) \big)\Bigr).
			\end{align*}
			Since $F_k^{(j)}$ converges in $L^2(\Omega, \HH)$ to $\hat F_k$ and $\psi^{(j)}_k$ converges uniformly to 
			$\hat \psi_k$ as $j \rightarrow \infty$
			we obtain together with Corollary \ref{stet_norm} that the right-hand side becomes arbitrary small for $j$ large enough.
			This finishes the proof.
	\end{proof}
	
	Up to now we have shown that for a given image sequence
	$I_0 = T, I_1,\dots,I_{K-1},$ $I_K = R \in L^2(\Omega,\HH)$ the problem
	\begin{equation*}
	\min_{\boldsymbol{\varphi} \in \mathcal{A}_\epsilon^K} \mathcal J(\vec I,\boldsymbol{\varphi})	.
	\end{equation*} 
	has a minimizer and that
	for given $\boldsymbol{\varphi} \in \mathcal{A}_\epsilon^K$, the problem
	\begin{equation*}
	\min_{\vec{I}\in \left( L^2(\Omega,\HH) \right)^{K-1}} \mathcal{J}(\vec I,\boldsymbol \varphi) 
	\quad \mbox{subject to } \quad I_0 = T, \; I_K = R
	\end{equation*}
	has a unique solution.
	Using these two results we can prove that a minimizer of $\boldsymbol{\mathcal{J}}$ in \eqref{d_path} exists.
	
	\begin{theorem}\label{main}
		Let $T,R\in L^2(\Omega,\HH)$ and $K\ge2$. 
		Then there exists a sequence $\hat{\vec I}\in L^2(\Omega,\HH)^{K-1}$ 
		minimizing $\boldsymbol{ \mathcal{J} }$.	
	\end{theorem}
	
	\begin{proof}
		Let 
		$\bigl\{ \vec{I}^{(j)}\bigr\}_{j\in\mathbb{N}}$, $\vec{I}^{(j)} \in L^2(\Omega,\HH)^{K-1}$
		be a minimizing sequence of  $\boldsymbol{\mathcal{J}}$. 
		Then $\boldsymbol{ \mathcal{J}}(\vec{I}^{(j)}) \le C$ for all $j\in\mathbb{N}$.
		By Corollary \ref{lem:ex:phi_eps}, we find for each $\vec{I}^{(j)}$ 
		a sequence of diffeomorphisms $\boldsymbol{\varphi}^{(j)}$ 
		such that
		$$
		\mathcal{J}\bigl(\vec{I}^{(j)} ,\boldsymbol{\varphi}^{(j)}\bigr) 
		\le 
		\mathcal{J} \bigl(\vec{I}^{(j)},\boldsymbol{\varphi}\bigr) 
		$$
		for all $\boldsymbol{\varphi}\in \mathcal{A}_\epsilon^K$.
		We know that $\lVert D^m \varphi_k^{(j)}\rVert_{L^2(\Omega)}^2<\frac{1}{\gamma}C$ 
		for all $j\in\mathbb{N}$ and $k = 1,\dots,K$. 
		As in the first part of the proof of Theorem \ref{lem:ex:phi}
		we conclude that \{ $\varphi_k^{(j)} \}_{j\in\NN}$ 
		is bounded in 
		$\left(W^{m,2}(\Omega)\right)^n$,
		so that there exists a subsequence 
		$\{ \varphi_k^{(j_l)} \}_{l\in\NN}$ 
		converging weakly in $\left(W^{m,2}(\Omega)\right)^n$ 
		and strongly in $\left( C^{1,\alpha}(\overline{\Omega}) \right)^n$,  $0 <\alpha < m-1-\frac{n}{2}$,
		to $\hat{\varphi}_k$. 
		Let us denote the whole subsequence again by $\{ \boldsymbol{\varphi}^{(j)} \}_{j\in\NN}$.
		
		Using Lemma \ref{lem:ex:seq} we can replace $\{ \vec{I}^{(j)} \}_{j\in\mathbb{N}}$
		by the image sequence $\{ \hat{\vec{I}}^{(j)} \}_{j\in\mathbb{N}}$,  $\hat{\vec{I}}^{(j)} \in L^2(\Omega,\HH)^{K-1}$ 
		minimizing $J_{\boldsymbol{\varphi}^{(j)}}$
		so that the energy ${\boldsymbol {\mathcal J}}$ does not increase. 
		By Corollary \ref{ConvergenceFK},
		we know that the sequence $\{ \hat{\vec{I}}^{(j)} \}_{j\in\mathbb{N}}$  converges in $(L^2(\Omega,\HH))^{K-1}$
		to $\hat {\vec I}$. 
		Thus,
		\begin{align}\label{dreieck}
		\lim_{j\to\infty}\sum_{k=1}^{K} \mathrm{d}_2^2 \Bigl(\hat{I}_{k-1}^{(j)} \circ \varphi_k^{(j)},\hat{I}_{k}^{(j)}\Bigr)
		&=	
		\sum_{k=1}^{K}\mathrm{d}_2^2 (\hat I_{k-1}\circ \hat{\varphi}_k,\hat I_{k}).			
		\end{align}
		By definition of $\boldsymbol{\mathcal{J}}$ we have
		\begin{align*}
		\boldsymbol{\mathcal{J}}(\hat{\vec I}) 
		&\le \mathcal{J} (\hat{\vec I},\hat{\boldsymbol{\varphi}}).
		\end{align*}
		Further by \eqref{dreieck}, as $W$ is lsc (see Part 4 of the proof of Theorem \ref{lem:ex:phi})
		and by construction of $\hat {\vec I}^{(j)}$, we obtain
		\begin{align*}
		\boldsymbol{\mathcal{J}}(\hat{\vec I}) &\le
		\liminf_{j\to\infty} \mathcal{J} ( \hat{\vec I}^{(j)}, \boldsymbol{\varphi}^{(j)} )
		\le 
		\liminf_{j\to\infty} \mathcal{J} ( \vec I^{(j)}, \boldsymbol{\varphi}^{(j)} )
		= 
		\liminf_{j\to\infty} \boldsymbol{\mathcal{J}} (\vec I^{(j)}).
		\end{align*}
		Thus, $\hat{\vec I}$ is a minimizer of $\boldsymbol{\mathcal{J}}$.
	\end{proof}

	\subsection{Model Specification}\label{sec:special_model}
	To find an image sequence $\vec I = (I_1,\ldots,I_{K-1})$  minimizing $\boldsymbol{\mathcal{J}}$, 
	we use an alternating minimization scheme.
	Starting with  $\vec I^{(0)}$ the
	deformation sequence $\boldsymbol{ \varphi}^{(j)}$ is computed by
	\begin{align}	
	\notag
	\boldsymbol{\varphi}^{(j)} 
	& \coloneqq \argmin_{\boldsymbol{\varphi}\in\diffeo_{\epsilon}^K} \mathcal{J}(\vec{I}^{(j-1)},\boldsymbol{\varphi})\\
	& =  \sum_{k=1}^K \argmin_{\varphi_k\in\diffeo_{\epsilon}} {\mathcal R}(\varphi_k; I_{k-1}^{(j-1)},I_k^{(j-1)})\label{non-couple}.
	\end{align}
	To obtain the image sequence $\vec I^{(j)}$ we compute
	\begin{align}
	\notag \vec I^{(j)} 
	& \coloneqq \argmin_{\vec I \in L^2(\Omega,\HH)^{K-1}} \mathcal{J}(\vec{I},\boldsymbol{\varphi}^{(j)})\\
	&= \argmin_{\vec I \in L^2(\Omega,\HH)^{K-1}} \sum_{k = 1}^{K}  d_2^2(I_{k-1}\circ \varphi_{k},I_{k}) \quad \mbox{subject to} \quad I_0 = T, \; I_K = R.
	\label{couple}
	\end{align}
	While in \eqref{non-couple} each of the $K$ minimization problems can be tackled separately, \eqref{couple} is a coupled system
	which can be solved using the approach in the proof of Theorem \ref{lem:uni:seq}.
	
	For computations, we  have to specify $W$ in ${\mathcal R} (\varphi) = {\mathcal R} (\varphi;T,R)$.
	Here we propose to use the linearized elastic potential, see, e.g., \cite[p. 99]{Mod2004}, 
	\begin{equation} \label{eq_spez_W}
	W(D\varphi) \coloneqq \mu \trace \left( (D\varphi_{\operatorname{sym}} -\mathbbm 1)^2 \right)
	+
	\frac{\lambda}{2}\trace \left((D\varphi_{\operatorname{sym}}-\mathbbm 1)\right)^2,
	\end{equation}
	where  $\mathbbm 1$ denotes the identity matrix.
	For this choice, we can reformulate the problem as finding a displacement vector field $v = (v_1,v_2)^\tT \colon \Omega\mapsto\RR^2$ 
	minimizing 
	\begin{align}
	\mathcal{R}(v) 
	\coloneqq
	\mathcal{S}(v) + \int_{\Omega} \gamma \lvert D^m v\rvert^2 + d^2 \bigr( T\left(x-v(x)\right),R(x) \bigl) \dx x, \label{eq:regu}
	\end{align}
	where 
	\begin{align}
	\mathcal{S}(v) &\coloneqq \int_{\Omega} \mu \trace \left( Dv_{\operatorname{sym}}^\tT \, Dv_{\operatorname{sym}} \right)
	+
	\frac{\lambda}{2}\trace \left( Dv_{\operatorname{sym}} \right)^2 + \eta  \lVert v\rVert_2^2 \, \dx x\label{eq:pot}.
	\end{align}
	For $\Omega \subset \mathbb R^2$, we have for example
	$$
	Dv_{\operatorname{sym}} = 
\begin{pmatrix} 
\partial_x v_1 & \frac12 (\partial_y v_1 + \partial_x v_2)\\
\frac12 (\partial_y v_1 + \partial_x v_2) & \partial_y v_2
\end{pmatrix},
$$
which is also known as the (Cauchy) strain tensor of the displacement $v$.
	Since the meaning is clear from the context, we use the same notation ${\mathcal R}$ when addressing $\varphi$ and $v$.
	To get the deformation we set $\varphi(x) \coloneqq x-v(x)$.

	The term \eqref{eq:regu} is a usual regularizer in the context of registration, see \cite{HM06,Mod2004,PPS17}. 
	However, after solving $K$ registration problems separately, 
	the resulting deformations are coupled in the subsequent step
	to find an optimal image sequence for these deformations.
	Using $\eta>0$ in~\eqref{eq:pot}
	together with some conditions on the distance of subsequent images $I_k$ we can handle
	the condition $\det(D\varphi)\ge\epsilon$ as the following remark shows.
	\begin{remark}\label{rem:feasible}
		1. Assume $\gamma,\eta>0$. Clearly, it holds 
		\begin{equation}\label{eq:rem1}
			\min_{\varphi\in\diffeo_{\epsilon}} {\mathcal R}(\varphi; T, R) \leq {\mathcal R}(\operatorname{id}; T, R) = \mathrm{d}^2_2(T,R).
		\end{equation}
	In a similar way as in the proof of Theorem \ref{lem:ex:phi}, 
	we can use the Gagliardo-Nirenberg inequality  together with the compact embedding of $W^{m,2}(\Omega)$ into $C^{1,\alpha}(\overline \Omega)$ 
	to conclude
		\begin{equation*}
		\Vert \varphi - \operatorname{id} \Vert_{C^{1,\alpha}(\overline \Omega)} 
		\leq C 
		\Vert \varphi - \operatorname{id} \Vert_{W^{m,2}(\Omega)} 
		\leq C \big(\Vert \varphi - \operatorname{id} \Vert_{L^2(\Omega)} + \Vert D^m( \varphi - \operatorname{id} )\Vert_{L^{2}(\Omega)} \big).
		\end{equation*}
		Thus, with the regularization~\eqref{eq:pot},
		the optimal solution $\hat \varphi$ of \eqref{eq:rem1} must fulfill
		\begin{equation*}
		\Vert \hat \varphi - \operatorname{id} \Vert_{C^{1,\alpha}(\overline \Omega)} \leq C \mathrm{d}_2(T,R),
		\end{equation*}
		where the constant $C$ is independent of $T$ and $R$. 
		Then we obtain that 
		$\Vert D \hat \varphi - \mathbbm{1} \Vert_{(L^\infty(\Omega))^{n,n}} \leq C \mathrm{d}_2(T,R)$. 
		If $\mathrm{d}_2(T,R) \leq C_{det}$ is sufficiently small, this implies together with the continuity of the determinant that 
		$\vert \det (D \hat \varphi)\vert \geq \epsilon$ .
		\\
		2.
		Now the argument needs to be extended to the whole problem
		\begin{equation*}
			\min_{\boldsymbol{\varphi}\in\diffeo_{\epsilon}^K, \vec I \in L^2(\Omega,\HH)^{K-1}} \sum_{k=1}^{K} {\mathcal R}(\varphi_k; I_{k-1}, I_{k})
		\end{equation*}
		with given template image \[I_0 = T\] and reference image \[I_K = R.\]
		 For this problem we use the initialization 
		$\tilde I_k = \gamma_{\overset{\frown}{T,R}}(\frac{k}{K})$ with the geodesic $\gamma_{\overset{\frown}{T,R}}$. 
		We conclude 
		\begin{equation*}
		\min_{\boldsymbol{\varphi}\in\diffeo_{\epsilon}^K, \vec I \in L^2(\Omega,\HH)^{K-1}} \sum_{k=1}^{K} {\mathcal R}(\varphi_k; I_{k-1}, I_{k})
		\leq \sum_{k=1}^{K} {\mathcal R}(\operatorname{id}; \tilde I_{k-1}, \tilde I_{k})
		= \frac{1}{K} d_2^2(T, R).
		\end{equation*}
		For every summand it holds ${\mathcal R}(\varphi_k; I_{k-1}, I_{k}) \leq \frac{1}{K} d_2^2(T,R)$ 
		which is smaller than $C_{det}$ 
		if $\mathrm{d}_2(T,R) \leq C_{det} \sqrt{K}$.
		This shows that for the optimal deformation it holds $\lvert \det(D \hat \varphi_k )\rvert \geq \epsilon$ if we use enough images in between.
	\end{remark}
	%
	\section{Minimization of the Space Discrete Model} \label{sec:d_model}
	In practice, we have to work in a spatially discrete setting.
	Dealing with digital images, we propose a finite difference model.
	We have already used such a model as basis of a face colorization method in \cite{PPS17}.
	In the rest of this paper, we restrict our attention to two-dimensional images 
	$  T , R \colon\grid\to\HH$ defined on the (primal) image grid 
	$\grid\coloneqq \{1,\ldots,n_1\} \times \{1,\ldots,n_2\} $. 
	We discretize the integrals appearing in our space continuous functionals 
	on the integration domain 
	$\overline \Omega\coloneqq [\tfrac12,n_1 + \tfrac12] \times [\tfrac12,n_2 + \tfrac12]$
	by the midpoint quadrature rule, i.e., with pixel values defined on $\grid$.

	\begin{figure}[t]
		\centering
		\begin{TikZOrPDF}{grid}
				\newcommand\ly{5}
				\newcommand\lx{7}
				\begin{tikzpicture}[scale=1.7]	
				\draw[step=1,xshift = 0.5cm,yshift = 0.5cm,very thin] (0,0) grid +(\lx,\ly);
				\draw[very thick,xshift = 0.5cm,yshift = 0.5cm] (0,0) rectangle (\lx,\ly);	
				\draw (1,\ly) node[anchor=north] {$(1,\ly)$};  
				\draw (1,1) node[anchor=north] {$(1,1)$};
				\draw (\lx,\ly) node[anchor=north] {$(\lx,\ly)$};
				\draw (\lx,1) node[anchor=north] {$(\lx,1)$};

				\foreach \x in {1,...,\lx}	{ \foreach \y in {1,...,\ly}  {  			
						\draw[fill] (\x,\y) circle (0.02); }  }
				
				\foreach \x in {1,...,\the\numexpr\lx-1\relax}	{ \foreach \y in {1,...,\ly}  {  		
						\draw[xshift = 0.5cm,\ifnum\x<2
						thin\else%
						\ifnum\x>\the\numexpr\lx-2\relax%
						thin\else fill%
						\fi%
						\fi] (\x-0.04,\y-0.04) rectangle ++(0.08,0.08); 			
						\draw[xshift = 0.5cm,->] (\x,\y)-- ++(0.15,0); }}
				
				\foreach \x in {1,...,\lx}	{ \foreach \y in {1,...,\the\numexpr\ly-1\relax}  {
						\draw[yshift = 0.5cm,\ifnum\y<2
						thin\else%
						\ifnum\y>\the\numexpr\ly-2\relax%
						thin\else fill%
						\fi%
						\fi] (\x-0.04,\y-0.04) rectangle ++(0.08,0.08); 
						\draw[yshift = 0.5cm,->] (\x,\y)-- ++(0,0.15); }}
				
				\begin{customlegend}[anchor=north,legend cell align=left,
				legend entries={ 
					Grid points of $\mathcal{G}$,
					Grid points of $\mathcal{G}_1$,
					Grid points of $\mathcal{G}_2$},
				legend style={at={({\the\numexpr\lx+1\relax},{\the\numexpr\ly+1\relax})},font=\footnotesize}] 
				\addlegendimage{mark=*,draw=black, only marks}
				\addlegendimage{uxbox}
				\addlegendimage{uybox}
				\end{customlegend}
				
				\end{tikzpicture}
		\end{TikZOrPDF}
		\caption{Illustration of the staggered grid, where empty boxes mean zero movement\label{fig:grid}}
	\end{figure}
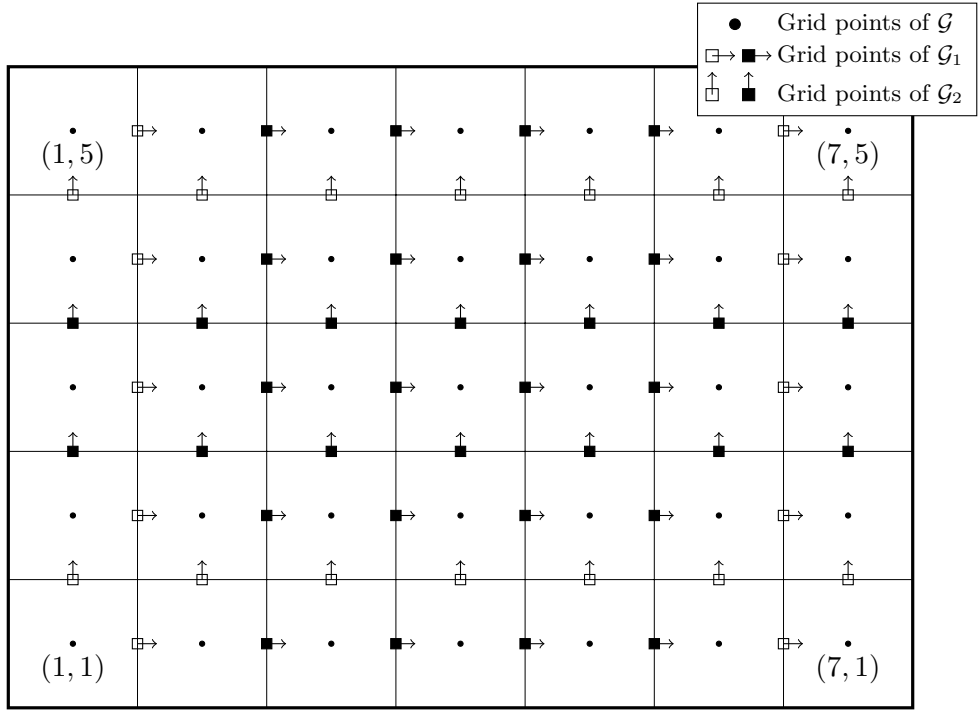

	\subsection{Computation of the Deformation Sequence}
	
	For discretizing the operators in \eqref{eq:pot} we work as usual on staggered grids.
	For an application of mimetic grid techniques in optical flow computation see also
	\cite{YSS07}.
	Let
	$$
	\grid_d \coloneqq \{\tfrac32,\ldots,n_1-\tfrac12\} \times \{\tfrac32,\ldots,n_2-\tfrac12\}
	$$
	be the (inner) dual grid, i.e., $\grid$ shifted by $\frac12$ in each direction,
	and
	$$
	\grid_1 \coloneqq \{\tfrac32,\ldots,n_1- \tfrac12\} \times \{1,\ldots,n_2\},\quad
	\grid_2 \coloneqq \{1,\ldots,n_1\} \times \{\tfrac12,\ldots,n_2- \tfrac32\}.
	$$
	Further, we consider $v = (v_1,v_2)^\tT$ with
	$v_1\colon   \grid_1 \rightarrow \mathbb R$ 
	and 
	$v_2\colon   \grid_2 \rightarrow \mathbb R$.
	In contrast to $\varphi(x) = x$, $x \in \partial \Omega$, 
	we allow flow along the boundary, i.e., $\langle v(x),n(x)\rangle_2 = 0$ for the outer normal $n(x)$ at $x\in\partial\mathcal{G}$. 
	In many applications features intersect with the boundary, see e.g., the green stripe in Fig.~\ref{fig:colormodels} 
	or the top part of Fig.~\ref{fig:cam}. Here movement along the boundary is reasonable even though the theorems from Section~\ref{sec:dm} do not necessarily hold true.
	To circumvent the gap to the theory we could embed the image into a larger domain by extending it with a constant value. After the computation the extension could be removed. But this leads to a higher computational effort as a larger image has to be processed 
	and the artificial boundary part might influence the original image.
	\\	
	Henceforth we use $\langle v(x),n(x)\rangle_2 = 0$, regarding the staggered grid the boundary conditions read
	\begin{align*}
	v_1(\tfrac32,x_2) &= v_1(n_1-\tfrac12,x_2) = 0,\quad x_2\in  \{1,\ldots,n_2\},\\
	v_2(x_1,\tfrac32) &= v_2(x_1,n_2-\tfrac12) = 0,\quad x_1\in  \{1,\ldots,n_1\}.
	\end{align*}
	See Fig.~\ref{fig:grid} for an illustration.
	We approximate $\partial_x v_1$ for $x = (x_1,x_2)^\tT \in \grid_p$ by
	\begin{equation*}
	(D_{1,x} v_1)(x_1,x_2) \coloneqq 
	\left\{
	\begin{array}{ll} 
	0& x_1 = 1,\\
	v_1(x_1+\tfrac{1}{2},x_2) -  v_1(x_1-\tfrac{1}{2},x_2) & x_1 = 2, \ldots, n_1-1,\\
	0& x_1 = n_1,
	\end{array}
	\right.
	\end{equation*}
	and $\partial_y v_1$ for $x_1 \in \{1,\ldots n_1-1\}$ and $x_2 \in \{1,\ldots,n_2-1\}$ by
	\begin{equation*}
	(D_{1,x_2} v_1)(x_1+\tfrac{1}{2},x_2+\tfrac{1}{2}) 
	=v_1(x_1 +\tfrac{1}{2},x_2+1) - v_1(x_1+\tfrac{1}{2},x_2),
	\end{equation*}
	and similarly for the derivatives of $v_2$.
	Finally, we obtain
	\begin{align*}
	\mathcal{S}(v) =& \sum_{ x \in \grid}
	\mu \left( (D_{1,x_1} v_1)^2(x) + (D_{2,x_1} v_2)^2(x) \right) 
	+ \frac{\lambda}{2} \big( (D_{1,x_1} v_1)(x) + (D_{2,x_2} v_2)(x) \big)^2\\
	&+ 
	\sum_{x \in \grid_d} \frac{\mu}{2} \big( (D_{1,x_2} v_1)(x) + (D_{2,x_1} v_2)(x) \big)^2.
	\end{align*}
		
	The deformations $v$ are living on grids, hence it is more convenient to use matrix-vector notation.	
	Let us first rewrite $\mathcal S (v)$.
	Using 
	$v_1 \coloneqq \left (v_1(x_1+\frac12,x_2) \right)_{x_1,x_2=1}^{n_1-1,n_2} \in \mathbb R^{n_1-1, n_2}$, 
	$v_2 \coloneqq \left(v_1(x_1,x_2+ \frac12) \right)_{x_1,x_2=1}^{n_1,n_2-1} \in \mathbb R^{n_1, n_2 - 1}$
	and
	\begin{align*}	
		D_{1,x_1} &\coloneqq
		\begin{pmatrix}
		0 &   & & & 0 \\
		-1& 1 & & & \\
		& & \ddots  & &\\
		0 & & 0&-1 & 1\\
		0 & & & 0& 0
		\end{pmatrix}\in\RR^{n_1\times n_1-1},
		\\ 
		D_{1,x_2} &\coloneqq
		\begin{pmatrix}
		-1& 1 &   &        &  &   & 0\\
		0& -1& 1 &        &  &   &\\
		&   &   & \ddots &  &   &\\
		&   &   &        &-1& 1 &0 \\
		0 &   &   &        &0 &-1 &1
		\end{pmatrix}\in\RR^{n_2-2\times n_2},
	\end{align*}
	and similarly 
	$D_{2,x_1} \in \RR^{n_1-2\times n_1}$ and
	$D_{2,x_2} \in \RR^{n_2\times n_2-1}$,
	we obtain
	\begin{align*}	
	\mathcal{S}(v) &= \mu \left( \|D_{1,x_1} v_1 \|_F^2 +  \| v_2 D_{2,x_2}^\tT \|_F^2 + \frac{1}{2} \|v_1 D_{1,x_2}^\tT +  D_{2,x_1} v_2 \|_F^2\right)\\&\quad
	+ \frac{\lambda}{2} \|D_{1,x_1} v_1 + v_2 D_{2,x_2}^\tT  \|_F^2 + \eta\lVert v_1\rVert_F^2+\eta\lVert v_2\rVert_F^2
	,
	\end{align*}
	where $\| \cdot \|_F$ denotes the Frobenius norm of matrices.
	Reshaping the $v_i$, $i=1,2,$ columnwise into vectors $\vec v_1 \in \mathbb R^{(n_1-1) n_2}$ and $\vec v_2 \in \mathbb R^{n_1 (n_2 - 1)}$
	(where we keep the notation) and using tensor products $\otimes$ of matrices in
	\begin{equation}\label{eq:difference}
	\begin{split}
	\vec D_{1,x_1} &\coloneqq \mathbbm 1_{n_2} \otimes D_{1,x_1},
	\qquad\ \ \ \vec D_{2,x_1} \coloneqq \mathbbm 1_{n_2-1 } \otimes D_{2,x_1}, \\
	\vec D_{1,x_2} &\coloneqq  D_{1,x_1} \otimes \mathbbm 1_{n_1-1 },
	\qquad\vec D_{1,x_2} \coloneqq  D_{2,x_2} \otimes \mathbbm 1_{n_1},
	\end{split}
	\end{equation}
	the regularizing term can be rewritten as
	\begin{equation}
	\mathcal{S}(\vec v) = \left\| \vec S  \vec v \right\|_2^2, 
	\quad 
	\vec S \coloneqq \begin{pmatrix}
	\sqrt{\mu} \, \vec D_{1,x_1} & 0\\
	0& \sqrt{\mu} \, \vec D_{2,x_2} \\
	\sqrt{\frac{\mu}{2}} \, \vec D_{1,x_2} &\sqrt{\frac{\mu}{2}} \, \vec D_{2,x_1}\\
	\sqrt{\frac{\lambda}{2}} \, \vec D_{1,x_1} &\sqrt{\frac{\lambda}{2}} \, \vec D_{2,x_2}\\
	\sqrt{\eta} \, \mathbbm{1}_{(n_1-1)n_2} & 0\\
	0 & \sqrt{\eta} \, \mathbbm{1}_{n_2(n_2-1)}
	\end{pmatrix}, \quad
	\vec v = \begin{pmatrix} \vec v_1\\ \vec v_2 \end{pmatrix}.\label{eq:reg}
	\end{equation}	
	
	For the discretization of the higher order term
		\begin{equation*}
		\lvert D^m \vec v\lvert^2 = \sum_{\lvert\alpha\rvert = m} \lVert \vec D^{\alpha}_{1} \vec v_1\Vert_2^2 + \lVert \vec D^{\alpha}_{2} \vec v_2\Vert_2^2,
		\end{equation*}
	we use finite difference matrices $\vec D^{\alpha}_{1}$, $\vec D^{\alpha}_{2}$ similar to~\eqref{eq:difference}. Setting $\vec D_1^m\coloneqq (\vec D^{\alpha}_{1})_{\lvert\alpha\rvert = m}$ and $\vec D_2^m\coloneqq (\vec D^{\alpha}_{2})_{\lvert\alpha\rvert = m}$ we can extend $\vec S$ from~\eqref{eq:reg} to 
	\begin{equation*}	
	\vec S \coloneqq \begin{pmatrix}
	\sqrt{\mu} \, \vec D_{1,x_1} & 0\\
	0& \sqrt{\mu} \, \vec D_{2,x_2} \\
	\sqrt{\frac{\mu}{2}} \, \vec D_{1,x_2} &\sqrt{\frac{\mu}{2}} \, \vec D_{2,x_1}\\
	\sqrt{\frac{\lambda}{2}} \, \vec D_{1,x_1} &\sqrt{\frac{\lambda}{2}} \, \vec D_{2,x_2}\\
	\sqrt{\eta} \, \mathbbm{1}_{(n_1-1)n_2} & 0\\
	0 & \sqrt{\eta} \, \mathbbm{1}_{n_2(n_2-1)}\\
	\sqrt{\gamma}\, \vec D_1^m & 0\\
	0 & 	\sqrt{\gamma}\, \vec D_2^m
	\end{pmatrix},
	\end{equation*}
	such that $\mathcal{S}(\vec v)$ includes all regularization terms.		
			
	To discretize the data term in  \eqref{eq:regu} we need to approximate $T\left( x - v(x)\right)$.
	Since $v$ is not defined on $\grid$ we use instead of $v$ its bilinear interpolation at the grid points, i.e., the averaged 
	version $Pv =(P_1v_1,P_2v_2) \coloneqq \grid\rightarrow \RR^2$  given by
	\begin{equation*}
	(P_1 v_1)(x_1,x_2) \coloneqq 
	\left\{
	\begin{array}{ll}
	0,& x_1=1,\\
	\tfrac12\left( v_1(x_1-\tfrac12,x_2) + v_1(x_1+\tfrac12,x_2) \right),& x_1=2,\ldots,n_1-1,\\
	0,& x_1=n_1,
	\end{array}
	\right.
	\end{equation*}
	and similarly for $P_2 v_2$ in $x_2$-direction. 
	In matrix-vector notation the averaging operator $P$ can be written as $Pv = (P_{n_1} v_1, v_2 P_{n_2}^\tT)$ using
	\begin{equation*}
		P_{n} = \frac{1}{2}	\begin{pmatrix}		
		1& 1 & & & 0\\
		& & \ddots  & &\\
		0 & & 0 & 1 & 1
		\end{pmatrix}\in\RR^{n,n+1}, \quad n \in\mathbb{N}.
	\end{equation*} 
	For the vectorized displacement $\vec v$ this becomes
	\begin{equation*}
	\vec P \vec v \coloneqq (\vec P_1\vec v_1,\vec P_2\vec v_2) \coloneqq \left( (\mathbbm{1}_{n_2} \otimes P_{n_1}) \vec v_1, (P_{n_2}\otimes \mathbbm{1}_{n_1}) \vec v_2 \right).
	\end{equation*} 		
	In general $x - Pv(x) \not \in {\mathcal G}$, so that the discrete image $ T$ has to be interpolated. 
	To this end, we use a counterpart of bilinear interpolation on manifolds, see e.g.~\cite{PFA06}. 
	It is based on a reinterpretation of the bilinear interpolation of real valued data points.	
	Let $f_{00},f_{01},f_{10},f_{11}\in\RR$ be the values at vertices of the unit cell. 
	Then the bilinear interpolation $f(x)$ at $x = (x_1,x_2)^\tT\in[0,1]^2$ is given by
	\begin{align*}
		f(x) 
		&= (1-x_1)(1-x_2)f_{00}+(1-x_1)x_2f_{01}+x_1(1-x_2)f_{10}+x_1x_2f_{11}\\
		&=\argmin_{f\in\RR}\Bigl\{(1-x_1)(1-x_2)(f-f_{00})^2+(1-x_1)x_2(f-f_{01})^2\\
		&\qquad\qquad\quad+x_1(1-x_2)(f-f_{10})^2+x_1x_2(f-f_{11})^2\Bigr\}.
	\end{align*}
	The latter formulation, which expresses the bilinear interpolation as mean, was applied e.g. in~\cite{PFA06} 
	to generalize bi- and trilinear interpolation to manifolds 
	using the Karcher mean with the appropriate weights:
	\begin{equation}\label{eq:karcher}
	\begin{split}
	f(x) &=\argmin_{f\in\HH}\Bigl\{(1-x_1)(1-x_2)d^2(f,f_{00})+(1-x_1)x_2d^2(f,f_{01})\\
	&\qquad\qquad\quad+x_1(1-x_2)d^2(f,f_{10})+x_1x_2d^2(f,f_{11})\Bigr\},
	\end{split}
	\end{equation}
	for $f_{00},f_{01},f_{10},f_{11}\in\HH$.
	In \cite{AGSW16}, it was shown that this bilinear interpolation  
	leads to $C^0(\Omega,\mathcal{H})$ images, which are $C^1(\Omega,\mathcal{H})$  at points $x = (x_1,x_2)^\tT$ 
	with neither $x_1\in\mathbb{Z}$ nor $x_2\in\mathbb{Z}$. 
	In the following, we write $T(x)\in\HH$ for the interpolated values of $T\colon\grid\to\HH$.

	In summary, the spatial discrete registration functional is given by
	\begin{equation*}\label{eq:disreg}
	\mathcal{R}(\vec v) \coloneqq \lVert\vec S \vec v\rVert_2^2 +\sum_{x\in\grid}d^2\bigl(T(x- (P v)(x)), R(x)\bigr).
	\end{equation*}
	
	To find a minimizer of this functional we want to apply a quasi-Newton method.
	To this end, we need the chain rule of the differential of
	two concatenated functions $F\colon\HH_1\to\HH_2$ and $G\colon\HH_0\to\HH_1$,
	$$
	D (F\circ G) (x) [\eta]\coloneqq DF\bigl(G(x)\bigr)\bigl[DG(x)[\eta]\bigr],
	$$
	see, e.g.~\cite[Chapter 4]{Lee12},
	and the definition of 	
	the Riemannian gradient of a function
	$F\colon\HH\to \mathbb R$, 
	$$
	\langle\nabla F(x),\eta\rangle_{x} \coloneqq DF(x)[\eta]  \quad \mbox{for all} \; \eta\in T_{x}\HH,
	$$
	where $\langle\cdot,\cdot\rangle_x$ denotes the Riemannian metric on the tangent space $T_x\HH$. 
	We start with the gradient of $T$. 
	Let $\lceil z\rceil$ be the smallest integer larger than $z$, and
	$\lfloor z\rfloor$ the largest integer smaller or equal than $z$.
	The derivative $\nabla_{x_1} T(x)\colon\RR^2\to T\mathcal{H}$ is approximated by
	\begin{equation}\label{eq:forwarddiff}
	\nabla_{x_1}   T(x) = \frac{1}{\lceil x_1\rceil-x_1}\log_{T(x)} T\bigl((\lceil x_1\rceil,x_2)\bigr).
	\end{equation}
	Note that for $T\bigl((\lceil x_1\rceil,x_2)\bigr)$ the minimization~\eqref{eq:karcher} reduces to solving
	\begin{align*}
	0 = (x_2-\lfloor x_2\rfloor)\log_{f}\bigl( T\bigl((\lceil x_1\rceil,\lceil x_2\rceil)\bigr)\bigr) + (x_2-\lceil x_2\rceil)\log_{f}\bigl(T\bigl((\lceil x_1\rceil,\lfloor x_2\rfloor)\bigr)\bigr)
	\end{align*}	
	for $f$. The solution is given analogously to~\eqref{getFF} by
	\begin{equation*}
	T\bigl((\lceil x_1\rceil,x_2)\bigr) 
	= \exp_{ T(\lceil x_1\rceil,\lfloor x_2\rfloor)}\Bigl((x_2- \lfloor x_2\rfloor)\log_{  T(\lceil x_1\rceil,\lfloor x_2\rfloor)}\bigl({  T(\lceil x_1\rceil,\lceil x_2\rceil)}\bigr)\Bigr).
	\end{equation*}
	The derivative in $x_2$-direction can be computed  analogously.
	
	Assuming  $T\in C^{1}(\Omega,\HH)$, we can calculate the gradient of $\mathcal{R}$ as
	\begin{align}\label{grad_1}
		\nabla_{\vec v}\mathcal{R}(\vec{v}) &= \nabla_{\vec{v}} \mathcal S(\vec v) + \nabla_{\vec v} \sum_{x\in\grid}d^2\bigl(T(x- (P v)(x)), R(x)\bigr)\\
		&= 2 \vec S^\tT\vec S \vec v + \vec G(\vec v).\notag
	\end{align}
	The gradient of the data term
	\begin{equation*}
	\vec G(\vec v) \coloneqq 
	\begin{pmatrix} 
	\nabla_{\vec v_1} \sum_{x\in\grid}d^2\bigl(T(x- (P v)(x)),R\bigr)\\
	\nabla_{\vec v_2} \sum_{x\in\grid}d^2\bigl(T(x- (P v)(x)),R\bigr)
	\end{pmatrix}
	\end{equation*} 
	is calculated by the chain rule 
	\begin{align}
	&\nabla_{\vec v_1} \Bigl(\sum_{x\in\grid} d^2\bigl(T(x- P\cdot ), R(x)\bigr)\Bigr)(\vec v)\notag\\
	&= -\vec P_1^{\tT} \nabla \Bigl( \sum_{x\in\grid} d^2\bigl(T(x+ (\cdot, -( P_2  v_2) (x))^{\tT}), R(x)\bigr)\Bigr)\left(-\vec P_1\vec v_1\right)\notag\\
	&= -\vec P_1^{\tT}  \Bigl( \sum_{x\in\grid} \nabla d^2\bigl(T(x+ (\cdot, -( P_2  v_2) (x))^{\tT}), R(x)\bigr)\left(-\vec P_1\vec v_1(x)\right)\Bigr).\label{eq:imagegrad}
	\end{align}
	Using the chain rule the ``inner'' derivative is given for $x\in\grid$, $v\in\RR^2$, and $\eta\in\RR$ by
	\begin{align*}
	D \Bigl(d^2\bigl(&T(x+ (\cdot,v_2)^\tT), R(x)\bigr)\Bigr)(v_1)[\eta] \\
	&= \Bigr\langle\nabla d^2\bigl(\cdot,R(x)\bigr)\bigl(T(x+v)\bigr),DT\bigr(x+(\cdot,v_2)^\tT\bigr)(v_1)[ \eta]\Bigr\rangle_{T(x+v)}\\
	&= \Bigl\langle\nabla d^2\bigl(\cdot,R(x)\bigr)\bigl(T(x+v)\bigr),\eta \, \nabla_{x_1} T(x+v)   \Bigr\rangle_{T(x+v)}\\
	&= \eta \, \langle-2\log_{T(x+v)}\bigl(R(x)\bigr),\nabla_{x_1} T(x+v)\rangle_{T(x+v)}\\
	&= \eta \, \nabla d^2\Bigl(T\bigl(x+ (\cdot,v_2)^\tT\bigr), R(x)\Bigr)(v_1).
	\end{align*}
	Plugging this into~\eqref{eq:imagegrad}, we obtain
	\begin{align*}
	\nabla_{\vec v_1} \Bigl(\sum_{x\in\grid} & d^2\bigl(T(x- P\cdot ), R(x)\bigr)\Bigr)(\vec v)\\
	&=\vec P_1^{\tT} \Bigl(2\bigl\langle\log_{T(x- (Pv)(x))} R(x),
	\nabla_{x_1} T\bigl(x- (Pv)(x)\bigr)\bigr\rangle_{T(x- (Pv)(x))}\Bigr)_{x\in\grid},
	\end{align*}
	and
	\begin{align*}
	\nabla_{\vec v_2} \Bigl(\sum_{x\in\grid} & d^2\bigl(T(x- P\cdot ), R(x)\bigr)\Bigr)(\vec v)
	=\\& \vec P_2^{\tT} \Bigl(2\bigl\langle\log_{T(x- (Pv)(x))} R(x),\nabla_{x_2} T\bigl(x- (Pv)(x)\bigr)\bigr\rangle_{T(x- (Pv)(x))}\Bigr)_{x\in\grid}.
	\end{align*}

	With the above derived gradient \eqref{grad_1} we can solve the registration problems by a gradient descent algorithm
	\begin{equation*}
	\vec v^{(l+1)} =\vec v^{(l)}-\tau \nabla_{\vec v} \mathcal{R}(\vec v^{(l)}),
	\end{equation*}
	with $\tau$ chosen, e.g.  by line search.
	As $\mathcal{R}$ is not convex we only get convergence of a subsequence to a critical point. 
	The result depends on the starting value, whose choice is described in Subsection~\ref{sec:multilvl}.
	Gradient descent algorithms are known to have bad convergence rates, hence we employ a quasi-Newton method. 
	The decent direction is given by
	\begin{equation*}
		\vec{D}(\vec{v}) = -\vec H^{-1} (\vec v^{(l)})\nabla_{\vec v}{\mathcal{R}}(\vec v^{(l)}),
	\end{equation*} 
	where $\vec H(\vec{v})$ is an approximation of the Hessian of $\mathcal{R}$. For the minimization of $\mathcal{R}$ we iterate
	\begin{equation*}
	\vec v^{(l+1)} =\vec v^{(l)}+\tau \vec{D}(\vec{v}^{(l)}),
	\end{equation*}
	with $\tau$ chosen by line search.
	The Hessian is approximated by 
	\begin{equation*}
	\vec{H}(\vec{v}) = 2\vec S^\tT\vec S +  \vec J(\vec{v})^{\tT} \vec J(\vec{v}),
	\end{equation*}
	where $\vec J(\vec{v}) \in \RR^{n_1n_2,((n_1+1)n_2+n_1(n_2+1))}$ 
	is the Jacobian of
	\begin{equation*}	
	d^2\bigl(T(x- (P\cdot )(x), R(x))\bigr) \colon\mathcal{G}_1\times \mathcal{G}_2 \to\RR.	\end{equation*}
	For vectorized images the Jacobian has the form
	\begin{align*}
	\vec J(\vec{v}) = 
	2 \biggl(&\diag\Bigl(\Bigl(\bigl\langle\log_{T(x-\vec P\vec{v}(x))}\vec R(x),\nabla_{x_1} T\bigl(x-\vec P\vec{v}(x)\bigr)\bigr\rangle_{T(x-\vec P\vec{v}(x))}\Bigr)_{x\in\grid}\Bigr)\vec P_1^\tT,\\&
	\diag\Bigl(\Bigl(\bigl\langle\log_{T(x-\vec P\vec{v}(x))}\vec R(x),\nabla_{x_2} T\bigl(x-\vec P\vec{v}(x)\bigr)\bigr\rangle_{T(x-\vec P\vec{v}(x))}\Bigr)_{x\in\grid}\Bigr)\vec P_2^\tT\biggr).
	\end{align*}
	
	\subsection{Computation of the Image Sequence}
	
	For a given displacement sequence $\vec V = (\vec v_1,\dots,\vec v_K)^\tT$,
	we can minimize 
	$J_{\boldsymbol{\varphi}} = J_{{\mathbbm 1}-\vec V}$ 
	by the construction given in the proof of Theorem \ref{lem:uni:seq},
	in particular formula \eqref{getFF}.
	This means we need to evaluate 
	geodesics between the initial $T$ and final image $R$ to obtain the minimizing image sequence. 
	In particular, for each $x\in\grid$ the geodesic 
	$\gamma_{\overset{\frown}{T\circ  \psi_0(x)}, R (x)}\subset\HH$ 
	has to be evaluated at the points $t_k(x)$, $k = 1,\dots,K-1$. 
	Here the grids $ \psi_k\colon \grid\to\Omega$, $k = 0,\dots,K,$ are derived by
	\begin{align*}
		 \psi_K(x) &= x, \quad
		 \psi_k(x) = \psi_{k+1}(x) - Pv_{k+1}\left( \psi_{k+1}(x)\right),
	\end{align*}
	where $ v_{k+1}\left(\psi_{k+1}(x)\right)$ 
	is obtained by bilinear interpolation on $\mathbb  R^2$. 
	The Jacobian of $\varphi_k = {\mathbbm 1}- v_k$ 
	needed in the computation is calculated using forward differences 
	and evaluated at the grid points $\psi_k(x)$ by bilinear interpolation.
	For each $x\in\grid$ and $k = 1,\dots,K-1$, we compute 
	$I_k \left(\psi_k(x)\right)$ 
	by evaluating the geodesics.
	Note that the intermediate images are  calculated at scattered points in $\Omega$ which are in general not on the grid. 
	Finally, the desired values $I_k(x) \in \HH$, $x\in\grid,$ 
	are obtained by linear scattered data interpolation of manifold-valued data.

	In the following remark we detail the convergence properties of the alternating minimization scheme~\eqref{non-couple} and~\eqref{couple} in the discrete setting.
	\begin{remark}[Convergence of the alternating minimization]~\\[-1.5\baselineskip]	
		\begin{description}
			\item[Deformation sequence] The computation of the deformation sequence is performed by an quasi-Newton method with appropriate step size. Hence, the deformations converge to a critical points of the functional and guarantee that the value of the functional does not increase.
			\item[Image sequence] The computation of the image sequence is based on the proof of Theorem~\ref{lem:uni:seq}, 
			which relies on continuous integrals. For the discrete setting, we have no proof that the functional value decreases
			which would imply a decrease of the whole functional $\mathcal{J}$ and the existence of a
			weakly convergent subsequence of images $\mathbf{I}^{(j_l)}$ in $(\HH^{n_1,n_2})^{K+1}$.
			However, in our numerical examples, we observed a decrease for non-degenerated deformation fields.		
		\end{description}
		\end{remark}
	\subsection{Multiscale Minimization Scheme}\label{sec:multilvl}
	Neither the energy $\mathcal J(\vec I, \boldsymbol{\varphi})$ nor the registration functional $\mathcal R(\vec v)$ is convex. 
	Hence an initialization close to the optimal solution is desirable.
	As usual in optical flow and image registration we apply a coarse-to-fine strategy. 
	First, we iteratively smooth our given images by convolution with a truncated Gaussian and downsample using bilinear interpolation.
	On the coarsest level we perform a single registration to obtain a deformation. 
	We apply a bilinear interpolation to construct a deformation on the finer level. Then we construct $\tilde K -1$, $\tilde K < K$ intermediate finer level images by
	\begin{equation}\label{eq:inter:img}
		I_k(x) = \exp_{T}\bigl(\frac{k}{\tilde K}\log_T (R\circ\varphi^{-1}\bigr)\bigl(x-\tfrac{k}{\tilde K}  Pv(x)\bigr),\quad \varphi(x) = x- Pv(x),
	\end{equation}
	where $T,R$ are the start and end images at the current level. The inversion of the mapping $\varphi\colon\grid\to\Omega$ is realized by linear scattered interpolation of the identification 
	\begin{equation*}
	\varphi^{-1}\bigr(\varphi_1(x_1,x_2),\varphi_1(x_1,x_2)\bigl) = (x_1,x_2).
	\end{equation*}
	Using this we obtain an initial image sequence on this level and start the alternating minimization 
	to obtain better deformations and intermediate images.

	Going to the next finer level by bilinear interpolation of the deformation and image sequence, 
	we construct more intermediate images by interpolating between neighboring ones similar to \eqref{eq:inter:img}.
	We repeat this procedure until we reach the maximal number $K$ of images and the finest level. 
	Going to finer level may increase the distance between subsequent images. 
	To keep the determinants of the optimal solution $\hat{\boldsymbol{\varphi}}$ bounded away from zero, 
	we can  adjust the number of intermediate images according to Remark~\ref{rem:feasible}.
	The multilevel strategy is sketched in Algorithm \ref{alg:morph}.
	
	\begin{algorithm}[htb]
		\caption{Morphing Algorithm (informal) }\label{alg:morph}
		\begin{algorithmic}[1]
			\State $T_0 \coloneqq  T, R_0 \coloneqq  R,\grid_0 \coloneqq \grid$
			\State create image stack $(T_l)_{l=0}^{\operatorname{lev}},(R_l)_{l=0}^{\operatorname{lev}}$ on $(\grid_l)_{l=0}^{\operatorname{lev}}$ by smoothing and downsampling
			\State solve \eqref{eq:disreg} for $T_{\operatorname{lev}},R_{\operatorname{lev}}$ to get $\tilde{v}$
			\State $l \to \operatorname{lev}-1$
			\State use bilinear interpolation to get $v$ on $\grid_{l}$ from $\tilde{v}$
			\State obtain $\tilde{K}_{l}$ images $\vec I_{l}^{(0)}$ from $T_{l},R_{l},v$ by \eqref{eq:inter:img}
			\While{$l\ge0$}
			\State find images $\tilde{\vec I}_l$ and deformations $\tilde{\vec v}_l$ minimizing \eqref{main} with initialization~$\vec I_l^{(0)}$
			\State $l\to l-1$
			\If{$l>0$}
			\State use bilinear interpolation to get $\vec I_l$ and $\vec  v_l$ on $\grid_l$
			\For{$k = 1,\dots,\tilde K _l$}
			\State calculate $\tilde{K}_l$ intermediate images between $I_{l,k-1},I_{l,k}$ with $v_{l,k}$ using~\eqref{eq:inter:img}
			\EndFor
			\EndIf
			\EndWhile	
			\vspace{0.4cm}\State $\vec I \coloneqq \vec I_0$
		\end{algorithmic}
	\end{algorithm}

	\section{Numerical Examples}\label{numerics}
	In this section, we present various proof-of-concept examples.
	While the minimization of \eqref{non-couple} and \eqref{couple}, as well as the muligrid scheme are implemented in \textsc{Matlab} 2017a, 
	the  manifold-valued image processing functions, like filtering, bilinear interpolation, and interpolation of scattered data are   
	implemented as part of the
	\href{http://www.mathematik.uni-kl.de/imagepro/members/bergmann/mvirt/}{\glqq Manifold-valued 
	Image Restoration Toolbox\grqq (MVIRT)}\footnote{open source, 
	available at~\href{http://www.mathematik.uni-kl.de/imagepro/members/bergmann/mvirt/}{www.mathematik.uni-kl.de/imagepro/members/bergmann/mvirt/}} 
	by Ronny Bergmann and Johannes Persch.
	The toolbox uses C\texttt{++} implementations of the basic manifold functions, like logarithmic and exponential maps, as well as the Karcher means, 
	which are imported into \textsc{Matlab} using \lstinline!mex!-interfaces with the GCC 4.8.4 compiler.\\
	In all examples, we set $m \coloneqq 3$ and  $\mu = \lambda = \gamma \eqqcolon\alpha$. The determinants of  $D\varphi_k$ in our numerical experiments stayed positive even using $\eta=0$.
	
	\subsection{Images in Different Color Spaces}\label{subsec:color}
	\begin{figure}
		\centering
		\newcommand{\imagewidth}{0.163\textwidth}
		\includegraphics[width = \imagewidth]{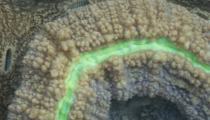}\,%
		\includegraphics[width = \imagewidth]{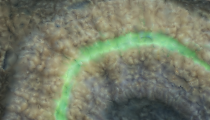}\,%
		\includegraphics[width = \imagewidth]{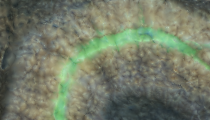}\,%
		\includegraphics[width = \imagewidth]{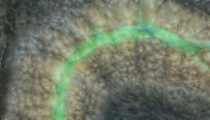}\,%
		\includegraphics[width = \imagewidth]{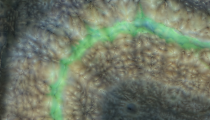}\,%
		\includegraphics[width = \imagewidth]{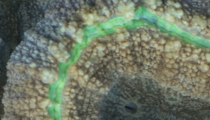}\\
		RGB color model
		\\[1ex]
		\includegraphics[width = \imagewidth]{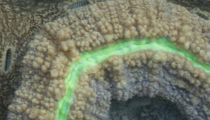}\,%
		\includegraphics[width = \imagewidth]{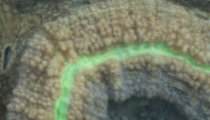}\,%
		\includegraphics[width = \imagewidth]{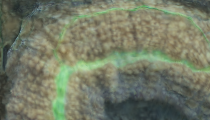}\,%
		\includegraphics[width = \imagewidth]{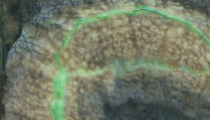}\,%
		\includegraphics[width = \imagewidth]{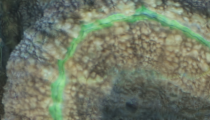}\,%
		\includegraphics[width = \imagewidth]{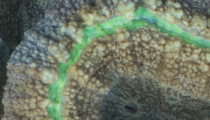}\\
		\includegraphics[width = \imagewidth]{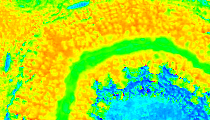}\,%
		\includegraphics[width = \imagewidth]{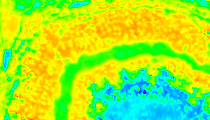}\,%
		\includegraphics[width = \imagewidth]{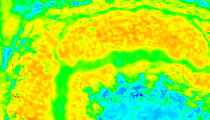}\,%
		\includegraphics[width = \imagewidth]{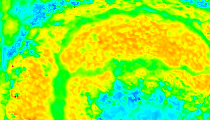}\,%
		\includegraphics[width = \imagewidth]{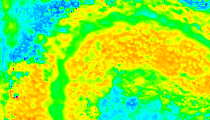}\,%
		\includegraphics[width = \imagewidth]{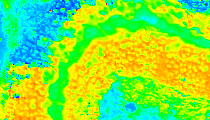}\\
		HSV color model with hue in the second row
		\\[1ex]
		\includegraphics[width = \imagewidth]{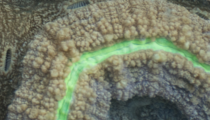}\,%
		\includegraphics[width = \imagewidth]{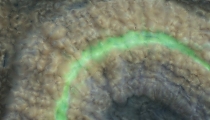}\,%
		\includegraphics[width = \imagewidth]{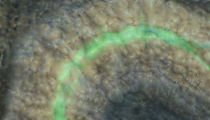}\,%
		\includegraphics[width = \imagewidth]{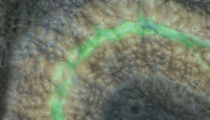}\,%
		\includegraphics[width = \imagewidth]{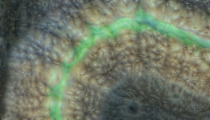}\,%
		\includegraphics[width = \imagewidth]{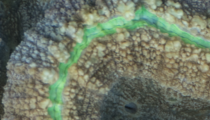}\\
		\includegraphics[width = \imagewidth]{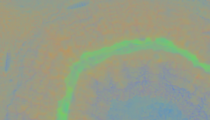}\,%
		\includegraphics[width = \imagewidth]{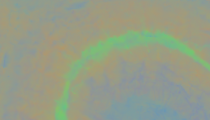}\,%
		\includegraphics[width = \imagewidth]{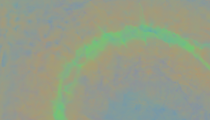}\,%
		\includegraphics[width = \imagewidth]{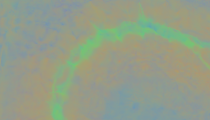}\,%
		\includegraphics[width = \imagewidth]{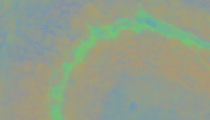}\,%
		\includegraphics[width = \imagewidth]{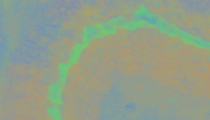}\\
		CB color model with chromaticity in second row
		\caption{Image path between two images of a sponge, using different color models. 
		}\label{fig:colormodels}
	\end{figure}
	First, we are interested in the morphing path of color images in different color spaces 
	having a nonlinear structure which is \emph{not} a Hadamard one.
	We compare 
	\begin{itemize}
	 \item[-] the linear RGB color space,\\[-1.25\baselineskip]
	  \begin{itemize}
	  	\item[] the image path is calculated using $\operatorname{lev}=7$, 
	  	where the image size is decreased to $75\%$ per level, $\alpha =0.00025$,
	  	and $\tilde K_6 = 3,\tilde K_5=2, \tilde K_4=1, \tilde K_l = 0, l=1,2,3$, i.e.,
	  	we decrease the number of new intermediate images while going to finer levels,
	  \end{itemize}
	 \item[-] the hue-saturation-value (HSV) color space, where the hue is phase valued, i.e., in $\mathbb S^1$,\\[-1.25\baselineskip]
	 \begin{itemize}
	 	\item[] the image path is calculated using $\operatorname{lev}=5$, 
	 	where the image size is decreased to $60\%$ per level, $\alpha =0.1$,
	 	and $\tilde K_5 = 3,\tilde K_4=2, \tilde K_3=1, \tilde K_l = 0, l=1,2$, and
	 \end{itemize} 
	 \item[-] the chromaticity-brightness (CB) color space,
	 where the chromaticity is $\mathbb S^2$-valued,\\[-1.25\baselineskip]
	 \begin{itemize}
	 	 \item[] the image path is calculated using $\operatorname{lev}=8$, 
	 	 where the image size is decreased to $75\%$ per level, $\alpha =0.00025$,
	 	 and $\tilde K_7 = 3,\tilde K_6=2, \tilde K_5=1, \tilde K_l = 0, l=1,\dots,4$. 
	 \end{itemize}
	\end{itemize}
	In Fig.~\ref{fig:colormodels} we see the image paths between to images of sponges. 
	We calculated the image paths with $25$ images, but only show the images $I_k, k = 0,4,9,14,19,24$.
	The intermediate images are blurred due to the bilinear interpolation, 
	which could be improved for the real-valued images, 
	but the computational cost would be very high for the manifold-valued images.
	
	The morphing for the HSV model looks strange when looking at the color image, 
	but is reasonable when considering the hue. In the hue channel the large yellow area is moving, while the green stripe is merging into and emerging out of the boundary. Here we work on the manifold 
	$\mathbb S ^1\times[0,1]^2\subset\mathbb S ^1\times\RR^2$ with the usual product metric. 
	Since the distances in $\mathbb S^1$ are larger than in the interval $[0,1]$, the hue 
	dominates the morphing. Changing the metric, i.e., the weights for the product metric,
	could lead to different results with more pleasing color images. Here we stick to 
	the usual choice to emphasize the importance of the metric. 
	
	The image path of the CB model is very similar to the RGB path for this image. 
	Looking at the chromaticity we see that on the right part of the image a small portion of the green color vanishes and appears again close by. 
	This effect could be reduced by lowering $\lambda$ and $\mu$, but then the deformations become close to irregular. 
	However, on the left side the movement of the green stripe looks smooth while the background changes as we expected.
	\subsection{Symmetric Positive Definite Matrices \texorpdfstring{$\SPD(n)$}{P(n)}}
	
	Next, we consider images with values 
	in the manifold of symmetric positive definite $n \times n$ matrices $\SPD(n)$
	with the affine invariant metric~\cite{PFA06}.
	\subsubsection*{Moving $\SPD(3)$ Rectangle}
	
		\begin{figure}
		\begin{subfigure}{\textwidth}
		\centering
		\foreach \n in {1,...,7}{\includegraphics[width = 0.135\textwidth]{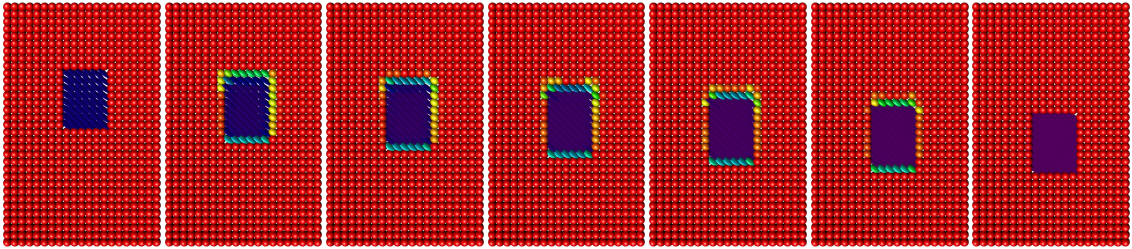}\,}
		\caption{Morphing path between two $\SPD(3)$ images with rectangular structures.}\label{fig:artspdpathhard}
		\end{subfigure}
		\begin{subfigure}{\textwidth}
		\foreach \n in {1,...,7}{\includegraphics[width = 0.135\textwidth]{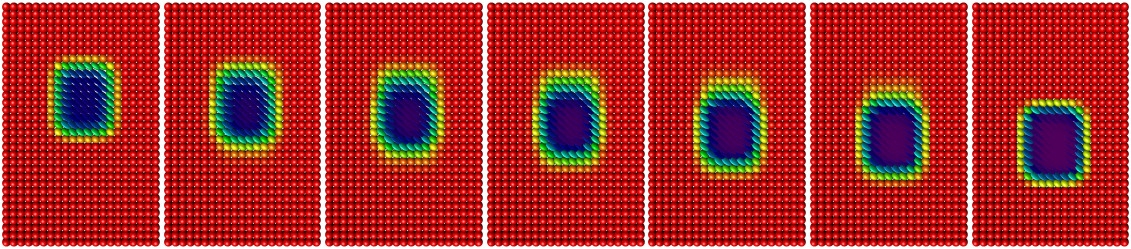}\,}	
		\caption{Morphing path between two $\SPD(3)$ images with smooth rectangular structures.}\label{fig:artspdpath}
		\end{subfigure}
		\caption{Comparison of the morphing of strong and smooth edges using the same set of parameters.}
	\end{figure}
	
	We start by computing a minimizing discrete path between simple synthetic images to see how edges are preserved.
	The template and reference images in Fig.~\ref{fig:artspdpathhard} consist of $3 I_3$ matrices in the background 
	and a rectangular part consisting of either the matrix
	\begin{equation*}
	A_{\mathrm{T}} = \begin{pmatrix}
		3 & 2 & 1\\2 & 4 &-1 \\1 & -1 & 2
	\end{pmatrix},\quad\text{or } A_{\mathrm{R}} = \exp_{3 I_3} (2\log_{3 I_3} A_{\mathrm{T}}),
	\end{equation*}
	which is moved downwards.
	The matrices are depicted as ellipsoids defined by their eigenvalues and eigenvectors. 
	For this image of size $21\times33$ we used $\operatorname{lev} =2$, 
	where the image size is decreased to $50\%$ per level and $\alpha = 1$.
	We calculated $5$ intermediate images between existing images on the new level
	to obtain $7$ images in total.
	The image path looks reasonable except for the smoothing of the rectangle in vertical morphing direction
	and at its right boundary. 
	The smoothing in the movement direction originates from the bilinear image interpolation model 
	used to obtain the intermediate images and the ``smoothness'' of the deformations. 
	It is possible to incorporate more sophisticated interpolation methods on manifolds, 
	but this involves higher computational cost. 
	The smoothing on the right side of the rectangle is an effect of the forward differences 
	used in the calculation of the discrete deformations~\eqref{eq:forwarddiff}. 
	This effect could also be reduced by a different discretization of the derivatives. 	
	However, for the similar images in Fig.~\ref{fig:artspdpath} with slightly smoothed edges our model performs well 
	and does not produce visible artifacts.
	\subsubsection*{Whirl $\SPD(2)$ Image}
	\begin{figure}
		\centering
		\includegraphics[width = 0.32\textwidth]{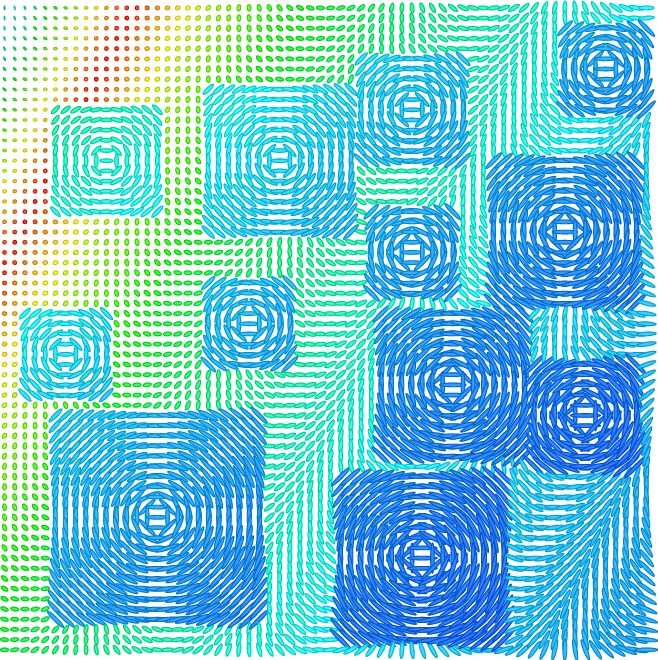}  
		\includegraphics[width = 0.32\textwidth]{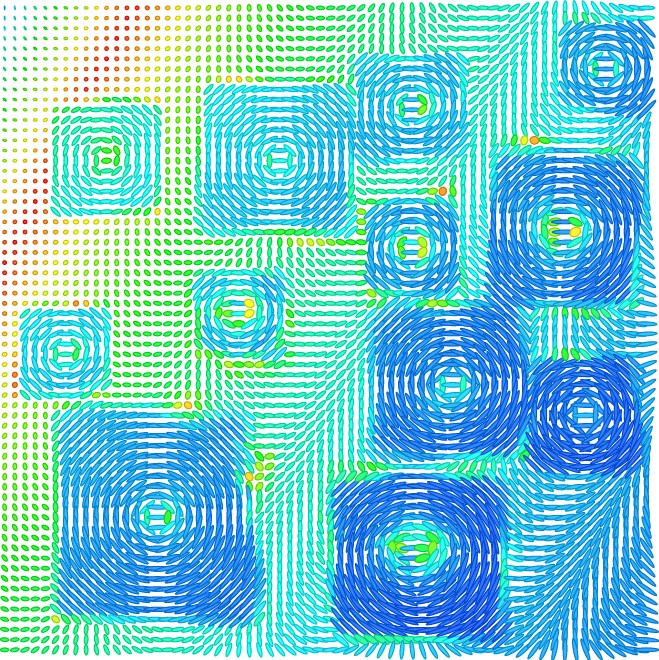} 
		\includegraphics[width = 0.32\textwidth]{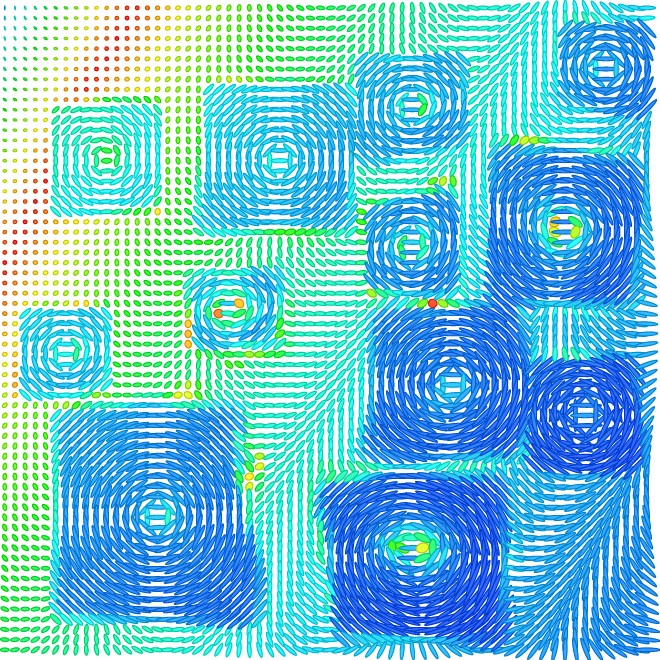} 
		
		\includegraphics[width = 0.32\textwidth]{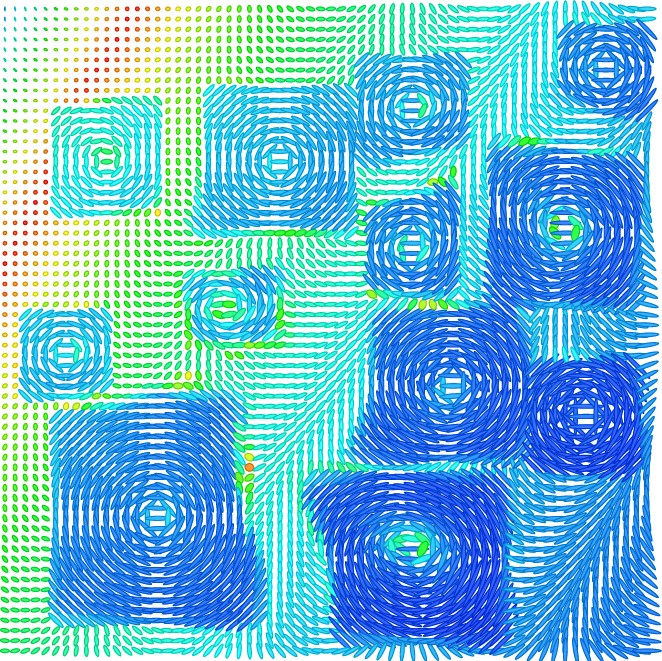} 
		\includegraphics[width = 0.32\textwidth]{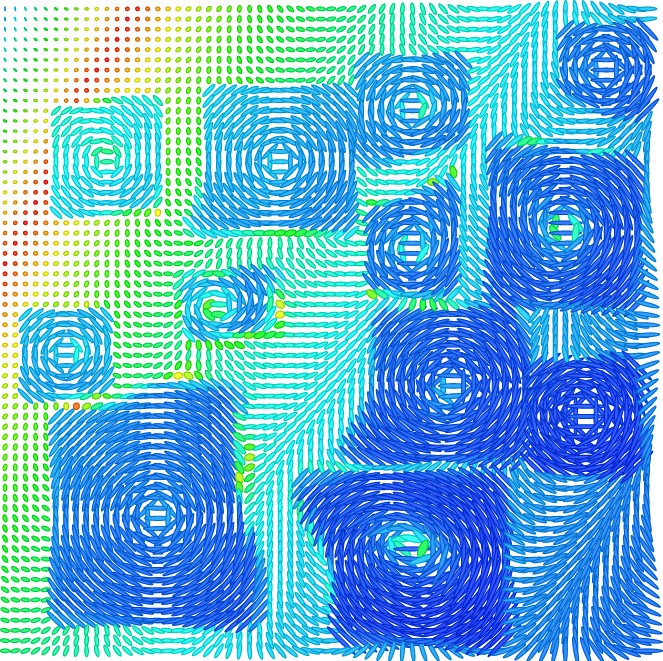} 
		\includegraphics[width = 0.32\textwidth]{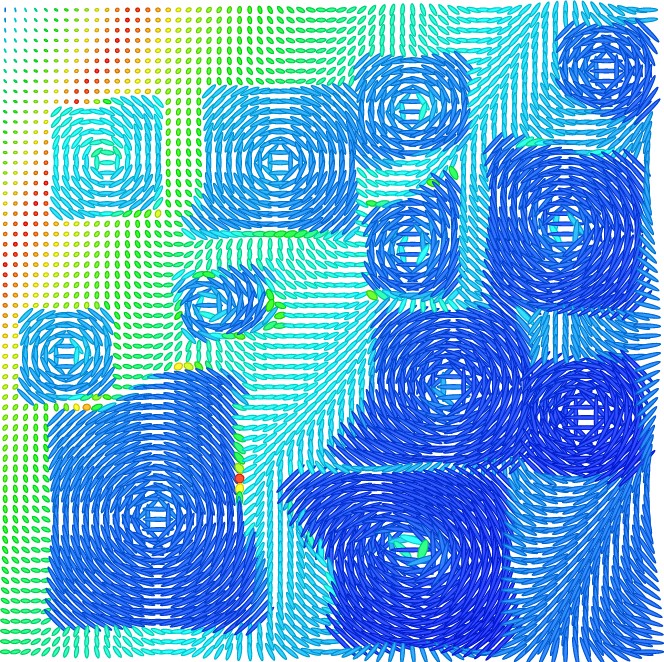} 
		\caption{Morphing path between two artificial $\SPD(2)$ images with whirl structures.}\label{fig:spd2}
	\end{figure}
	In Fig.~\ref{fig:spd2} we compute the discrete path between two $\SPD(2)$ images, 
	where we obtained the final one by deforming the start image and pushing its values further away from the identity. 
	The artificial deformation is more complicated as in the previous example. For this image of size $64\times 64$ we used $\operatorname{lev} = 4$, 
	where the image size is decreased to $75\%$ per level, $\alpha = 0.005,$
	$\tilde K_3 = 3,\tilde K_2=2, \tilde K_1=1$. 	
	Even though the deformation is more complicated than before, the path shows a reasonable transition from the starting to the final image.
	
	\subsubsection*{DT-MRI}
	\begin{figure}
		\centering
		\includegraphics[width = 0.50\textwidth]{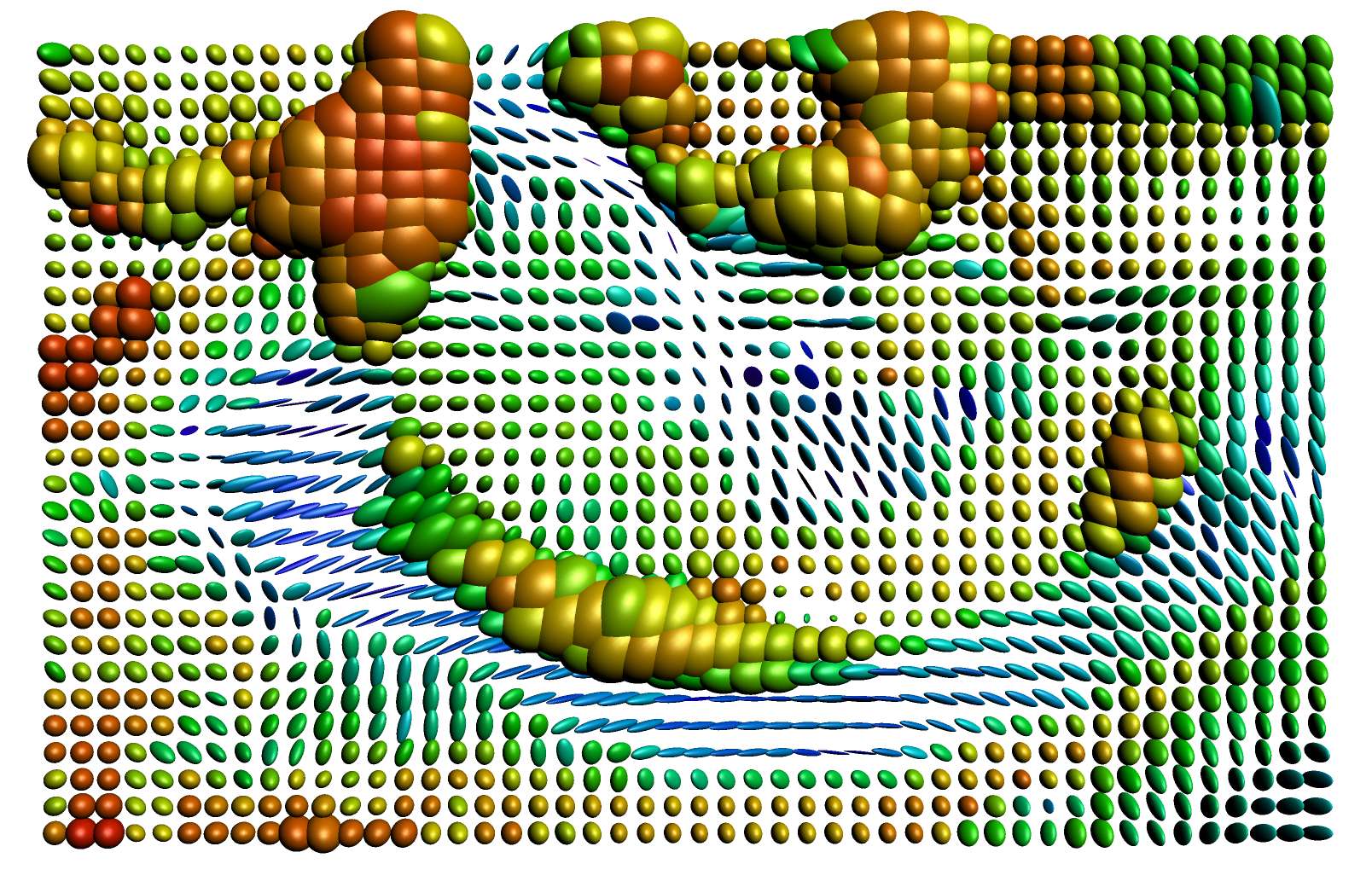}%
		\includegraphics[width = 0.50\textwidth]{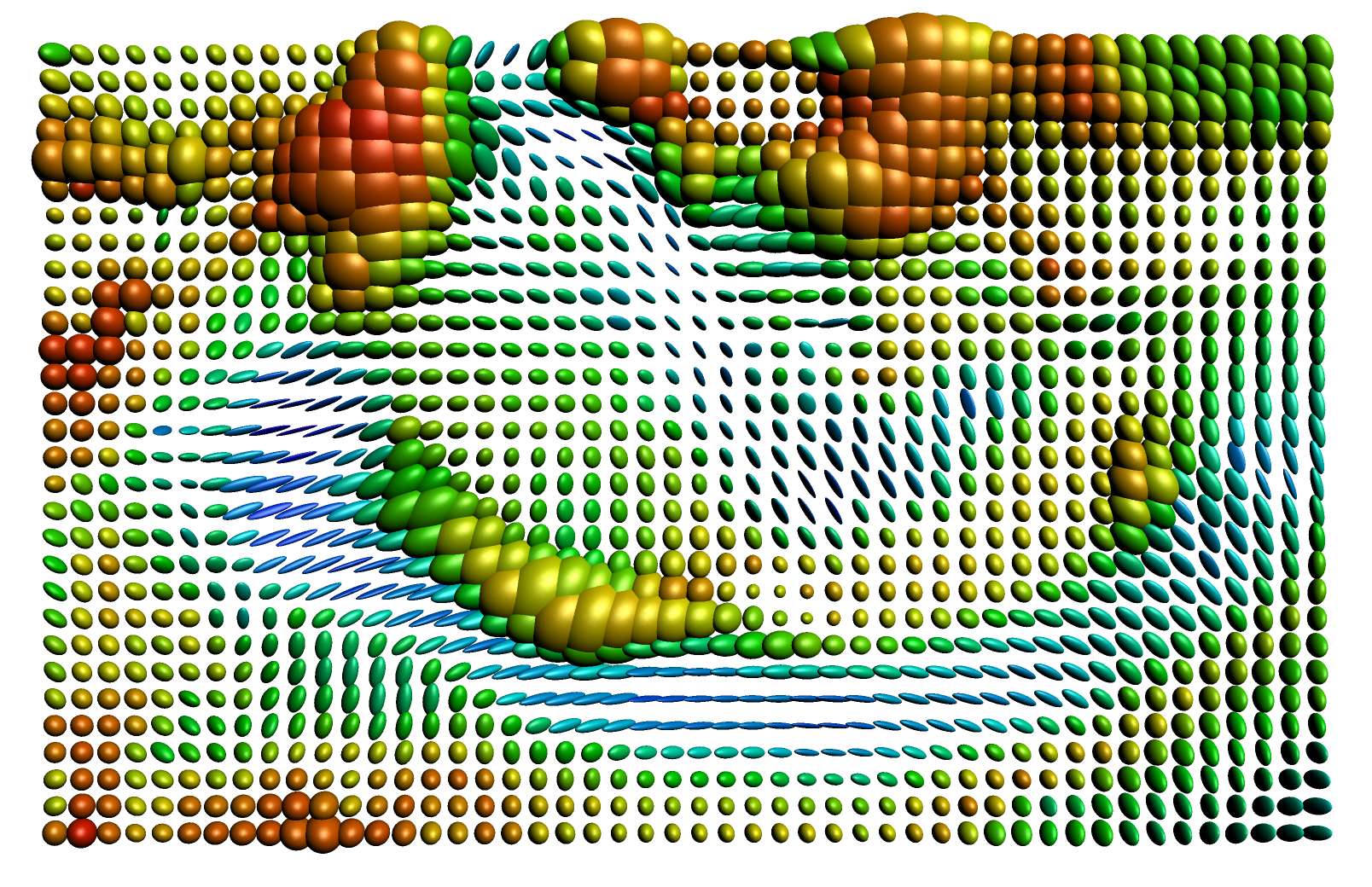}%
				
		\includegraphics[width = 0.50\textwidth]{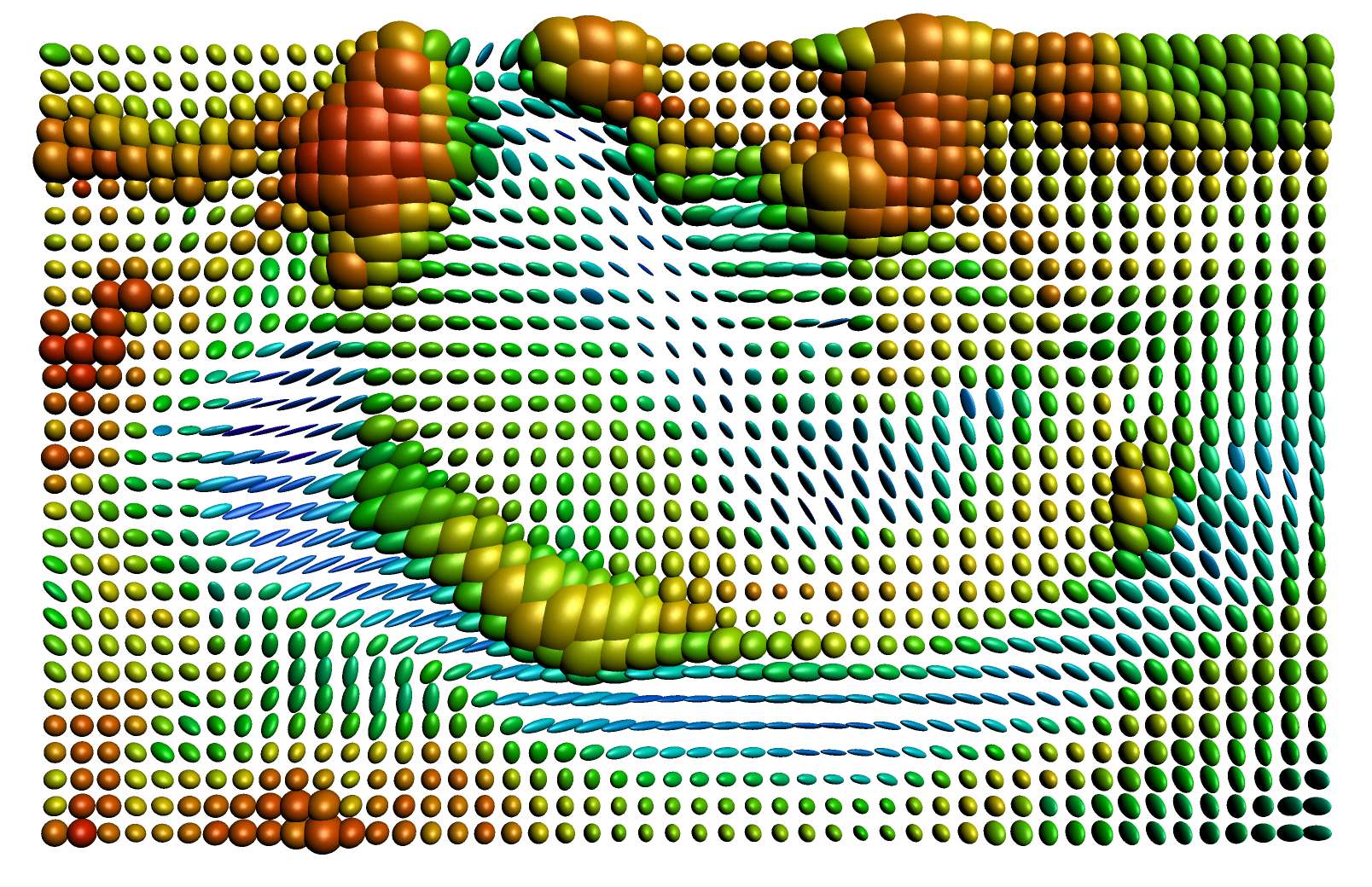}%
		\includegraphics[width = 0.50\textwidth]{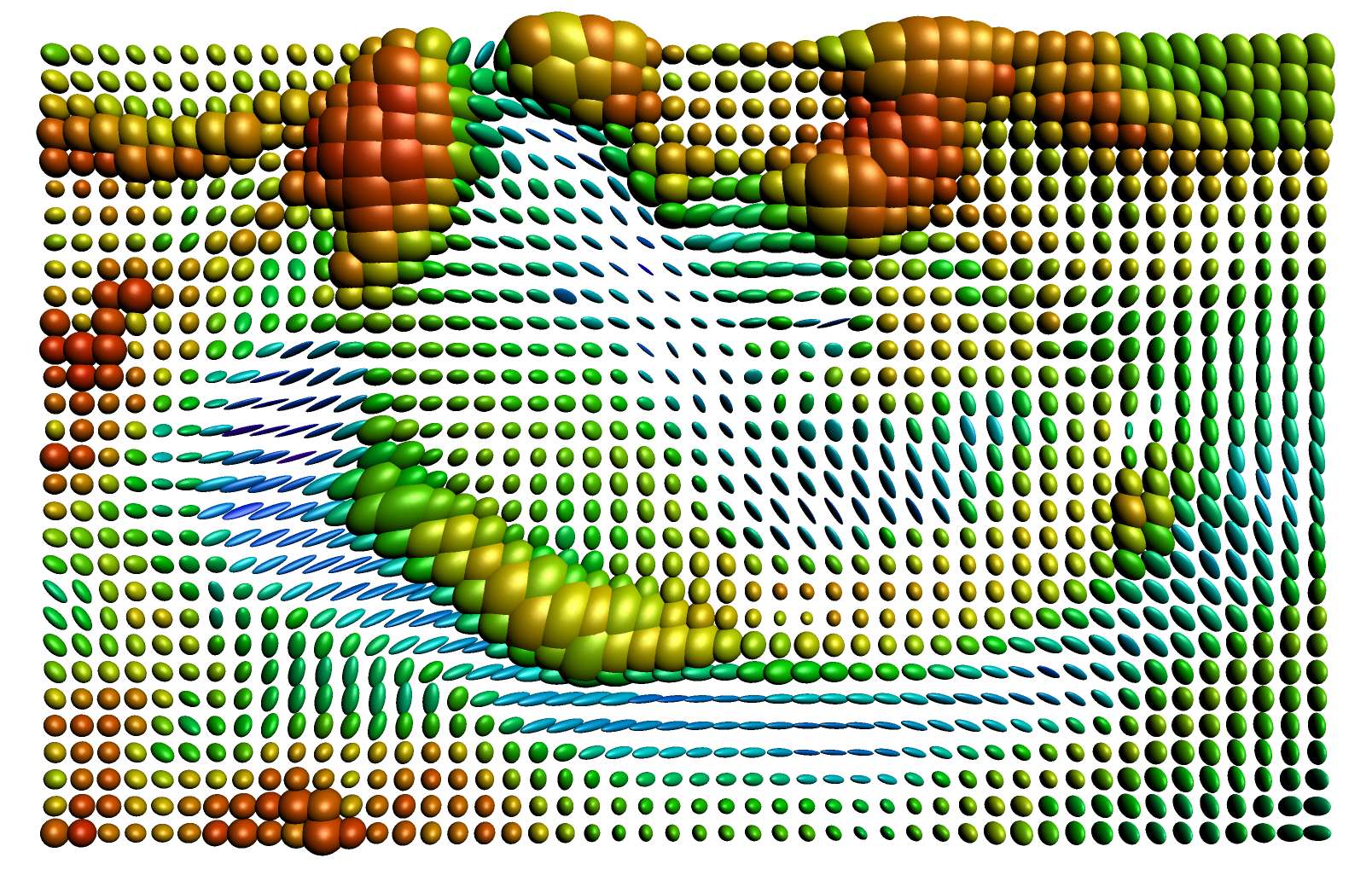}%
		
		\includegraphics[width = 0.50\textwidth]{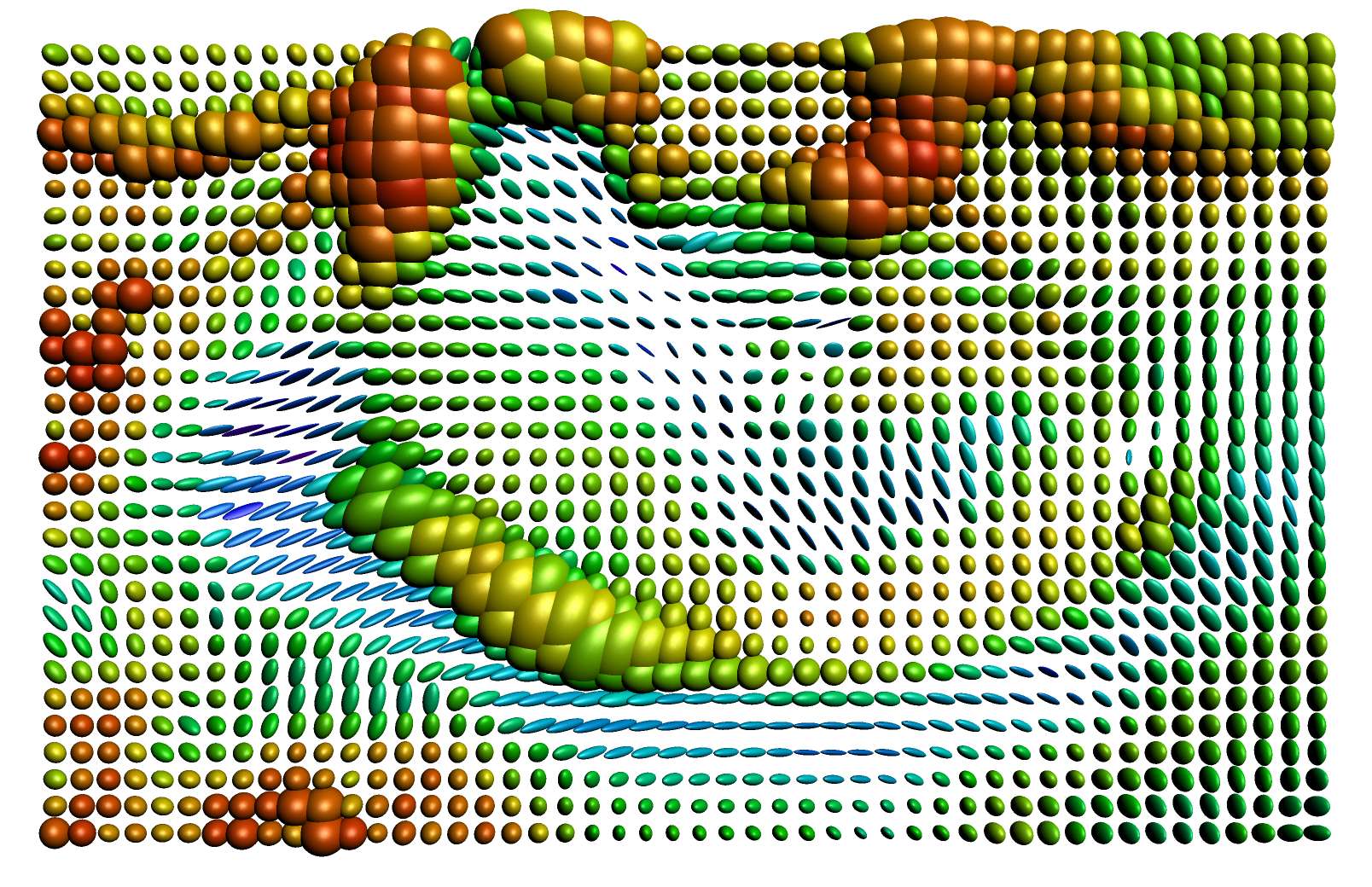}%
		\includegraphics[width = 0.50\textwidth]{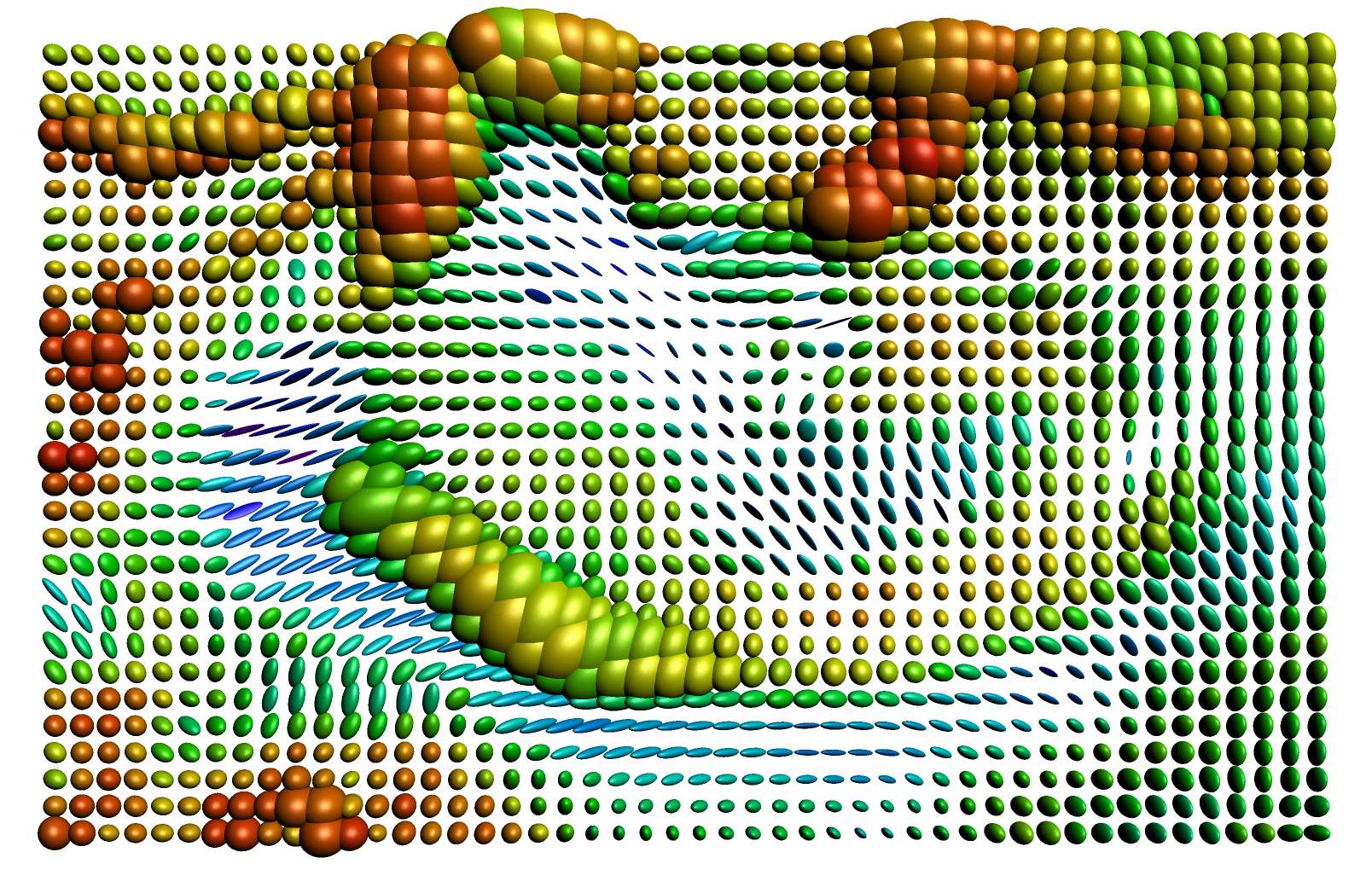}%
		
		\includegraphics[width = 0.50\textwidth]{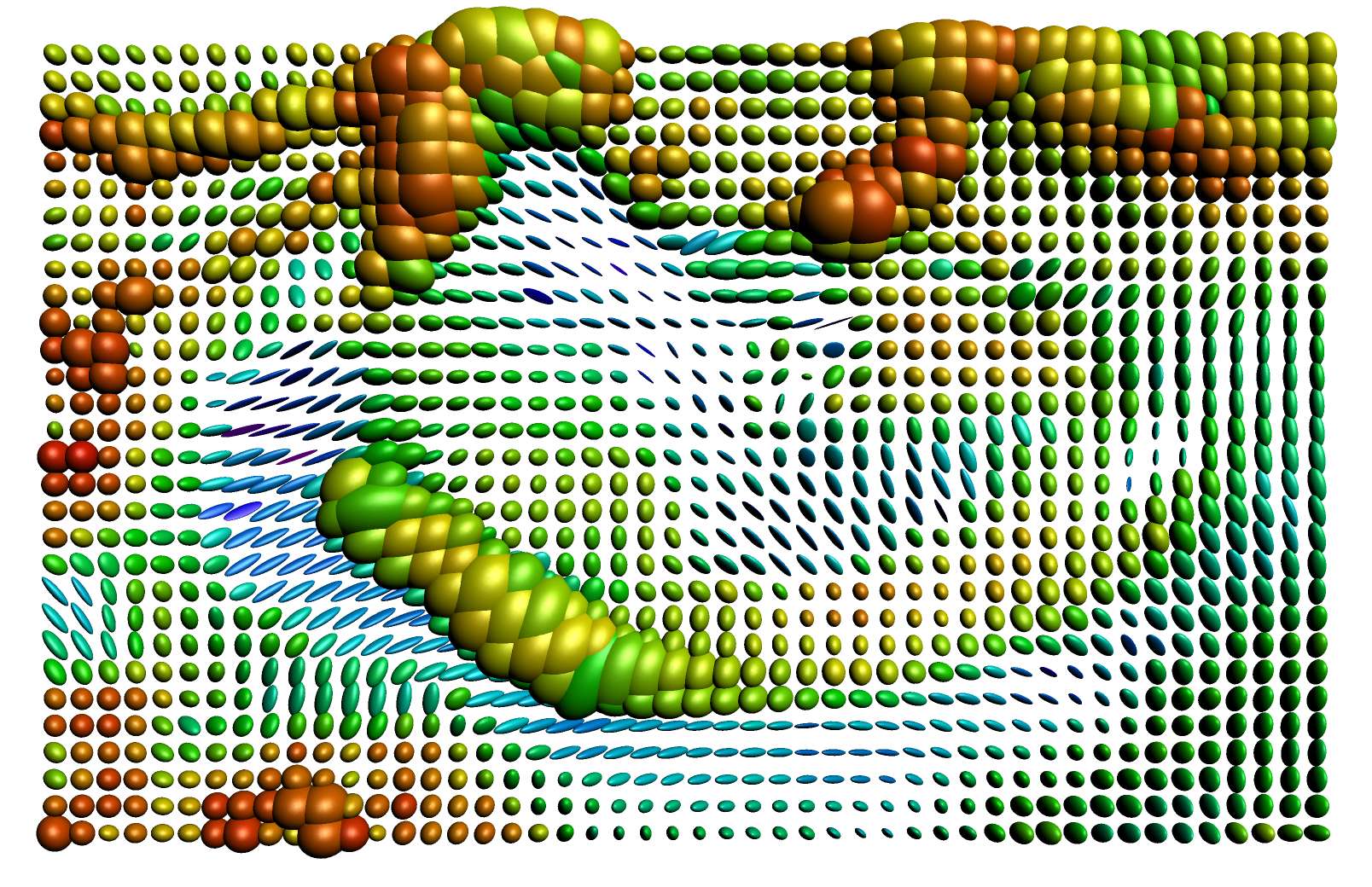}%
		\includegraphics[width = 0.50\textwidth]{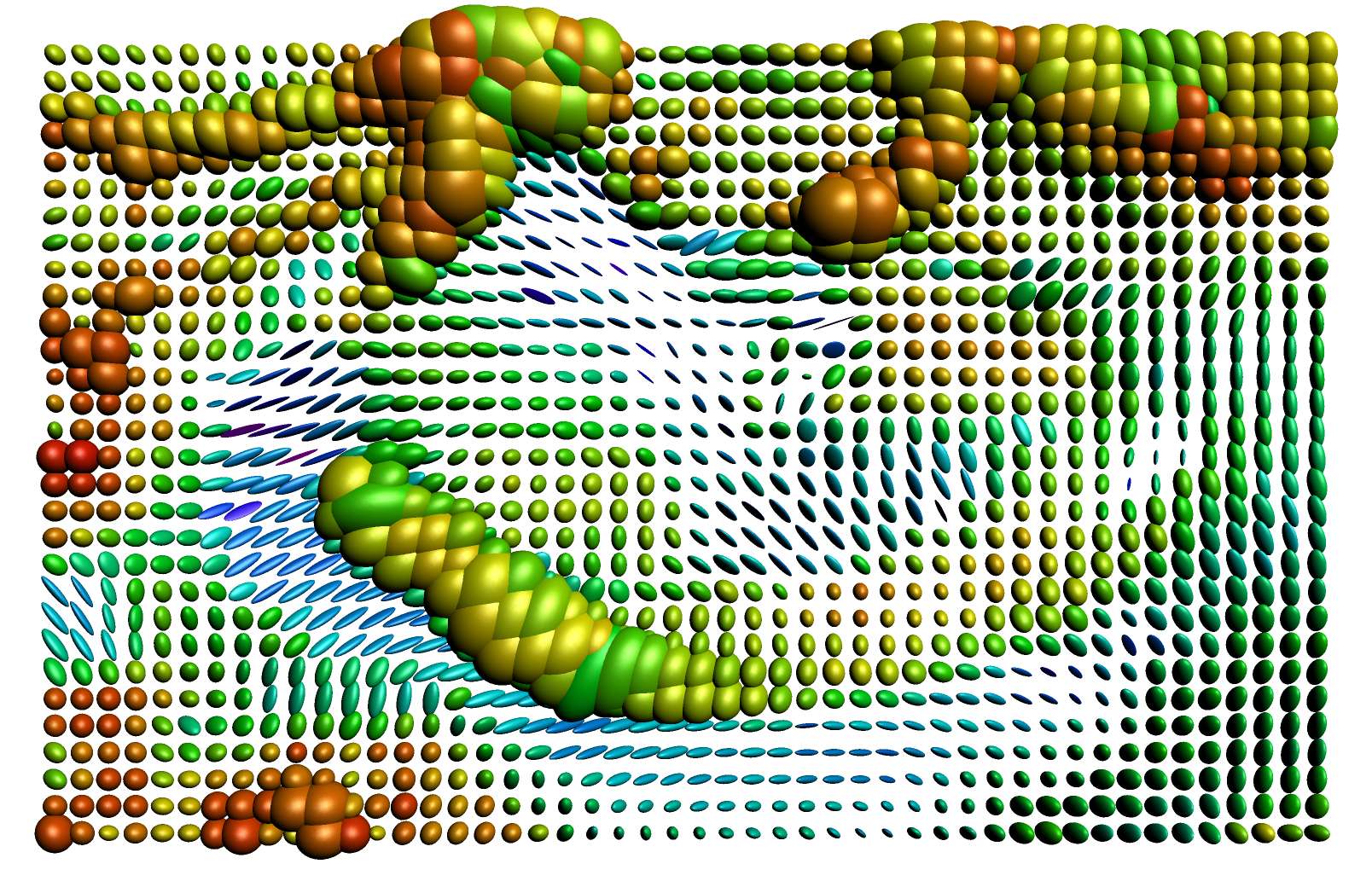}%
		\caption{Morphing path between a part of the YZ-slices 49 and 51 of the Camino dataset.}\label{fig:cam}
	\end{figure}
	The morphing path between two slices of the Camino\footnote{\url{http://camino.cs.ucl.ac.uk}}~\cite{camino} 
	is shown in Fig.~\ref{fig:cam}. As a preprocessing we inpainted the holes and slightly smoothed the slices 
	using a $\ell_2-\operatorname{TV}$ regularization with the Douglas-Rachford algorithm~\cite{BPS16}.
	The image path is calculated using $\operatorname{lev}=5$, 
	where the image size is decreased to $80\%$ per level, $\alpha =0.025$,
	and $\tilde K_4 = 2,\tilde K_3=\tilde K_2=1, \tilde K_1 = 0$, i.e.,
	we decrease the number of new intermediate images while going to finer levels. 
	Several interesting effects can be seen on the path, e.g., 
	the spot of large tensors on the right blends into the background, as there is no similar structure in the final image. 
	Further the stripe in the center moves a bit to the bottom left. While big structures on the top merge together, a small structure at the left boundary separates. 
	As our implementation of the spatial discrete setting involves many means, the images on the morphing path rather very smooth. 
	Hence a comparison with the original Camino slice  between the template 
	and the target image is not reasonable.
	
	\section{Conclusions} \label{sec:conclusions}
In this paper we have shown how the time discrete geodesic path model of~\cite{BER15}
can be generalized to the manifold-valued images.
Indeed, we have used a modified setting and have shown that at least
for finite dimensional Hadamard manifolds there exist minimizers for the space continuous model.

We have outlined the computational details for the smooth transition of discrete images with values on a manifold,
where the computations work for more general manifolds than Hadamard ones.
We have to clarify if the solution of the system of equations
from the space continuous setting, which is also used in the space discrete model
still leads to a decrease of the corresponding functional.
In all our numerical examples this was the case.
Furthermore, it is questionable if some theory can be generalized to other manifolds.
For examples of manifold-valued images, see e.g.~\cite{BCHPS16}.
A generalization of Theorem~\ref{lem:uni:seq} to manifolds with arbitrary curvature 
seems to be possible if the images live in compact and convex subsets of the manifold. 
Compactness ensures the existence of a minimizer in the $L^2(\Omega,\HH)$. 
Convexity is used for the uniqueness of shortest geodesics.

The most interesting question is if the model converges for $K \rightarrow \infty$
to some meaningful functional such that it can be interpreted as time discretization
of some geodesic path in the space of manifold-valued images. Note that we have dealt 
with the Mosco convergence in Hadamard spaces in~\cite{BMS2016}.


\appendix
\section{Gagliardo-Nirenberg inequality}\label{app:gn}
	\begin{theorem}[Gagliardo-Nirenberg \cite{Nir1966}]\label{th:gn}
		Let $\Omega\subset\RR^n$ be a bounded domain satisfying the cone property. 
		For $1\le q, r\le\infty$, 
		suppose that $f$ belongs to $L^q(\Omega)$ and its derivatives of order $m$ to $L^r(\Omega)$. 
		Then for the derivatives $D^j f$, $0\le j < m$ the following inequalities hold true
		with constants $C_1,C_2$ independent of $f$:
		\begin{equation*}
			\lVert D^j f \rVert_{L^p(\Omega)}\le C_1\lVert D^m f \rVert_{L^r(\Omega)}^a\lVert f\rVert_{L^q(\Omega)}^{1-a}+C_2\lVert f \rVert_{L^q(\Omega)},
		\end{equation*}
		where
		$
			\frac{1}{p} = \frac{j}{n}+a\Bigl(\frac{1}{r}-\frac{m}{n}\Bigr)+(1-a)\frac{1}{q}
		$
		for all $a \in [\frac{j}{m}, 1]$, except for the case $1<r<\infty$ and $m-j-\frac{n}{r}$ is a nonnegative integer, in which the inequality 
		only holds true for  $a \in [\frac{j}{m},1)$.
	\end{theorem}
	\begin{remark}\label{th:gnr}
		For $p=q=r=2$ the inequality simplifies to
		\begin{align*}
			\lVert D^j f \rVert_{L^2(\Omega)} &\le C_1\lVert D^m f \rVert_{L^2(\Omega)}^{\frac{j}{m}}\lVert f\rVert_{L^2(\Omega)}^{1-\frac{j}{m}}+C_2\lVert f \rVert_{L^2(\Omega)}
			\\
			&\leq C_1\lVert D^m f \rVert_{L^2(\Omega)} + \big(C_1 + C_2 \big) \lVert f \rVert_{L^2(\Omega)},
		\end{align*}
		where the second inequality follows by estimating the product with the maximum of both factors.
	\end{remark}

\subsection*{Acknowledgments} 
We are grateful to the anonymous referee for the valuable comments.
We want to thank M. Rumpf for fruitful discussions.
	Many thanks to Miroslav Ba{\v c}{\'a}k for 
	bringing the book of Bridson and Haefliger to our attention
	which points out that closed and bounded sets in locally convex Hadamard spaces are compact.
	Funding by the German Research Foundation (DFG) within the project STE 571/13-1 
	and with\-in the Research Training Group 1932,
	project area P3, is gratefully acknowledged.

	\bibliographystyle{abbrv}
	\bibliography{references}	

\begin{thebibliography}{10}

\bibitem{AGSW16}
P.-A. Absil, P.-Y. Gousenbourger, P.~Striewski, and B.~Wirth.
\newblock Differentiable piecewise-{B}{\'e}zier surfaces on {R}iemannian
  manifolds.
\newblock {\em SIAM Journal on Imaging Sciences}, 9(4):1788--1828, 2016.

\bibitem{alt2002}
H.~W. Alt.
\newblock {\em Lineare Funktionalanalysis: Eine anwendungsorientierte
  Einf\"uhrung}, volume~6.
\newblock Springer, Berlin, 2002.

\bibitem{Bac14}
M.~Ba{\v{c}}{\'a}k.
\newblock {\em Convex {A}nalysis and {O}ptimization in {H}adamard {S}paces},
  volume~22 of {\em De Gruyter Series in Nonlinear Analysis and Applications}.
\newblock De Gruyter, Berlin, 2014.

\bibitem{BMS2016}
M.~Ba{\v{c}}{\'a}k, M.~Montag, and G.~Steidl.
\newblock Convergence of functions and their {M}oreau envelopes on {H}adamard
  spaces.
\newblock {\em Journal of Approximation Theory}, Accepted.

\bibitem{Ball1981}
J.~M. Ball.
\newblock Global invertibility of {S}obolev functions and the interpenetration
  of matter.
\newblock {\em Proceedings of the Royal Society of Edinburgh: Section A
  Mathematics}, 88(3-4):315--328, 1981.

\bibitem{BMTY2005}
M.~F. Beg, M.~I. Miller, A.~Trouv{\'e}, and L.~Younes.
\newblock Computing large deformation metric mappings via geodesic flows of
  diffeomorphisms.
\newblock {\em International Journal of Computer Vision}, 61(2):139--157, 2005.

\bibitem{BCHPS16}
R.~Bergmann, R.~H. Chan, R.~Hielscher, J.~Persch, and G.~Steidl.
\newblock Restoration of manifold-valued images by half-quadratic minimization.
\newblock {\em Inverse Problems and Imaging}, 10(2):281–304, 2016.

\bibitem{BPS16}
R.~Bergmann and G.~S. J.~Persch.
\newblock A parallel {D}ouglas–{R}achford algorithm for restoring images with
  values in symmetric {H}adamard manifolds.
\newblock {\em SIAM Journal on Imaging Sciences}, 9(3):901--937, 2016.

\bibitem{BER15}
B.~Berkels, A.~Effland, and M.~Rumpf.
\newblock Time discrete geodesic paths in the space of images.
\newblock {\em SIAM Journal on Imaging Sciences}, 8(3):1457--1488, 2015.

\bibitem{BH1999}
M.~R. Bridson and A.~H\"afliger.
\newblock {\em Metric Spaces of Non-Positive Curvature}.
\newblock Grundlehren der Mathematischen Wissenschaften 319, Springer-Verlag,
  Berlin, 2004.

\bibitem{CJ2001}
G.~E. Christensen and H.~J. Johnson.
\newblock Consistent image registration.
\newblock {\em IEEE Transactions on Medical Imaging}, 20(7):568--582, 2001.

\bibitem{CRM96}
G.~E. Christensen, R.~D. Rabbitt, and M.~I. Miller.
\newblock Deformable templates using large deformation kinematics.
\newblock {\em IEEE Transactions on Image Processing}, 5(10):1435--1447, 1996.

\bibitem{camino}
P.~A. Cook, Y.~Bai, S.~Nedjati-Gilani, K.~K. Seunarine, M.~G. Hall, G.~J.
  Parker, and D.~C. Alexander.
\newblock {C}amino: Open-source diffusion-{MRI} reconstruction and processing.
\newblock In {\em 14th Scientific Meeting of the International Society for
  Magnetic Resonance in Medicine}, page 2759, Seattle, WA, USA, 2006.

\bibitem{DGM98}
P.~Dupuis, U.~Grenander, and M.~I. Miller.
\newblock Variational problems on flows of diffeomorphisms for image matching.
\newblock {\em Quarterly of Applied Mathematics}, 56(3):587--600, 1998.

\bibitem{Effland17}
A.~Effland.
\newblock {\em Discrete Riemannian Calculus and A Posteriori Error Control on
  Shape Spaces}.
\newblock Disertation, University of Bonn, 2017.

\bibitem{FM2003}
B.~Fischer and J.~Modersitzki.
\newblock Curvature based image registration.
\newblock {\em Journal of Mathematical Imaging and Vision}, 18(1):81--85, 2003.

\bibitem{FJSY09}
M.~Fuchs, B.~J\"uttler, O.~Scherzer, and H.~Yang.
\newblock Shape metrics based on elastic deformations.
\newblock {\em Journal Mathematical Imaging and Vision}, 35(1):68--102, 2009.

\bibitem{HM06}
E.~Haber and J.~Modersitzki.
\newblock A multilevel method for image registration.
\newblock {\em SIAM Journal on Scientific Computing}, 27(5):1594--1607, 2006.

\bibitem{HBDHRSSU2007}
J.~Han, B.~Berkels, M.~Droske, J.~Hornegger, M.~Rumpf, C.~Schaller, J.~Scorzin,
  and H.~Urbach.
\newblock Mumford--{S}hah model for one-to-one edge matching.
\newblock {\em IEEE Transactions on Image Processing}, 16(11):2720--2732, 2007.

\bibitem{HJS12}
Y.~Hong, S.~Joshi, M.~Sanchez, M.~Styner, and M.~Niethammer.
\newblock Metamorphic geodesic regression.
\newblock In {\em In Proc. of International Conference on Medical Image
  Computing and Computer-Assisted Intervention}, volume 7512 of Lecture Notes
  in Computer Science, pages 197--205, 2012.

\bibitem{Jost97}
J.~Jost.
\newblock {\em Nonpositive Curvature: Geometric and Analytic Aspects}.
\newblock Lectures in Mathematics ETH Z\"urich. Birkh\"auser Verlag, Basel,
  1997.

\bibitem{Lee12}
J.~M. Lee.
\newblock {\em Introduction to Smooth Manifolds}, volume 218 of {\em Graduate
  Texts in Mathematics}.
\newblock Springer, New York, 2012.

\bibitem{MTY02}
M.~I. Miller, A.~Trouv{\'e}, and L.~Younes.
\newblock On the metrics and euler-lagrange equations of computational anatomy.
\newblock {\em Annual Review of Biomedical Engineering}, 4(1):375--405, 2002.

\bibitem{MTY15}
M.~I. Miller, A.~Trouv{\'e}, and L.~Younes.
\newblock Hamiltonian systems and optimal control in computational anatomy: 100
  years since d'arcy thompson.
\newblock {\em Annual Review of Biomedical Engineering}, 17:447--509, 2015.

\bibitem{MY2001}
M.~I. Miller and L.~Younes.
\newblock Group actions, homeomorphisms, and matching: A general framework.
\newblock {\em International Journal of Computer Vision}, 41(1-2):61--84, 2001.

\bibitem{Mod2004}
J.~Modersitzki.
\newblock {\em Numerical Methods for Image Registration}.
\newblock Oxford University Press on Demand, 2004.

\bibitem{Mod2009}
J.~Modersitzki.
\newblock {\em {FAIR}: Flexible Algorithms for Image Registration}.
\newblock SIAM, Philadelphia, 2009.

\bibitem{Nir1966}
L.~Nirenberg.
\newblock An extended interpolation inequality.
\newblock {\em Annali Della Scuola Normale Superiore di Pisa-Classe di
  Scienze}, 20(4):733--737, 1966.

\bibitem{PFA06}
X.~Pennec, P.~Fillard, and N.~Ayache.
\newblock A {R}iemannian framework for tensor computing.
\newblock {\em International Journal of Computer Vision}, 66(1):41--66, 2006.

\bibitem{PPS17}
J.~Persch, F.~Pierre, and G.~Steidl.
\newblock Exemplar-based face colorization using image morphing.
\newblock {\em Journal of Imaging}, 3(4):ArtNum 48, 2017.

\bibitem{RY16}
C.~L. Richardson and L.~Younes.
\newblock Metamorphosis of images in reproducing kernel {H}ilbert spaces.
\newblock {\em Advances in Computational Mathematics}, 42(3):573--603, 2016.

\bibitem{Rudin1964}
W.~Rudin.
\newblock {\em {A}nalysis}.
\newblock McGraw-Hill, 1964.

\bibitem{RW13}
M.~Rumpf and B.~Wirth.
\newblock Discrete geodesic calculus in shape space and applications in the
  space of viscous fluidic objects.
\newblock {\em SIAM Journal on Imaging Sciences}, 6(4):2581--2602, 2013.

\bibitem{RW15}
M.~Rumpf and B.~Wirth.
\newblock Variational time discretization of geodesic calculus.
\newblock {\em IMA Journal of Numerical Analysis}, 35(3):1011--1046, 2015.

\bibitem{Smythe1990}
D.~B. Smythe.
\newblock A two-pass mesh warping algorithm for object transformation and image
  interpolation.
\newblock Technical report, ILM Technical Memo 
  Department, Lucasfilm Ltd, 1990.

\bibitem{Tro95}
A.~Trouv{\'e}.
\newblock An infinite dimensional group approach for physics based models in
  pattern recognition.
\newblock {\em International Journal of Computer Vision}, 1995.

\bibitem{Tro98}
A.~Trouv{\'e}.
\newblock Diffeomorphisms groups and pattern matching in image analysis.
\newblock {\em International Journal of Computer Vision}, 28(3):213--221, 1998.

\bibitem{TY2005a}
A.~Trouv\'e and L.~Younes.
\newblock Local geometry of deformable templates.
\newblock {\em SIAM Journal of Mathematical Analysis}, 37(2):17--59, 2005.

\bibitem{TY2005b}
A.~Trouv\'e and L.~Younes.
\newblock Metamorphoses through {L}ie group action.
\newblock {\em Foundations in Computational Mathematics}, 5(2):173--198, 2005.

\bibitem{wolberg1990}
G.~Wolberg.
\newblock {\em Digital Image Warping}, volume 10662.
\newblock IEEE Computer Society Press, Los Alamitos, CA, 1990.

\bibitem{Wolberg1998}
G.~Wolberg.
\newblock Image morphing: a survey.
\newblock {\em The Visual Computer}, 14(8):360--372, 1998.

\bibitem{Younes2010}
L.~Younes.
\newblock {\em Shapes and Diffeomorphisms}.
\newblock Springer-Verlag, Berlin, 2010.

\bibitem{YSS07}
J.~Yuan, C.~Schn\"orr, and G.~Steidl.
\newblock Simultaneous higher order optical flow estimation and decomposition.
\newblock {\em SIAM Journal of Scientific Computing}, 29(6):2283--2304, 2007.

\end{thebibliography}
\end{document}